\documentclass[final,3p]{elsarticle}
\usepackage{lineno,hyperref}
\usepackage{amsmath}
\usepackage{amssymb}
\usepackage[linesnumbered,ruled,vlined]{algorithm2e}
\usepackage{caption}
\usepackage{mathtools}
\usepackage{changes}
\usepackage{multirow}
\usepackage{verbatim}
\usepackage{mathrsfs}
\usepackage{graphicx}
\usepackage{subcaption}
\usepackage{amsmath,amsthm,bm,mathrsfs}
\theoremstyle{plain}
\newtheorem{theorem}{Theorem}[section]

\newtheorem{proposition}[theorem]{Proposition}

\theoremstyle{definition}

\theoremstyle{remark}

\modulolinenumbers[5]

\journal{ArXiv.org}

\begin{document}

\begin{frontmatter}

\title{A New Generalized Low-Rank Cholesky Factor ADI Algorithm for Large-Scale Stein Equations}

\author[uz]{Umair~Zulfiqar\corref{mycorrespondingauthor}}
\cortext[mycorrespondingauthor]{Corresponding author}
\ead{umair@yangtzeu.edu.cn}
\address[uz]{School of Electronic Information and Electrical Engineering, Yangtze University, Jingzhou, Hubei, 434023, China}
\begin{abstract}
The low-rank alternating direction implicit (ADI) method is an efficient solver for large-scale Stein equations with low-rank solutions. This paper shows that, as in the continuous-time Lyapunov equation case, the low-rank Cholesky factor ADI (LRCF-ADI) method for Stein equations implicitly performs $\mathcal{H}_2$-pseudo-optimal model order reduction for discrete-time systems. This observation leads to an automatic shift-generation strategy, allowing LRCF-ADI to select subsequent shifts without user intervention.

The standard LRCF-ADI method requires shifts outside the unit circle. We generalize the method to allow shifts anywhere in the complex plane, including on the unit circle. This extension enables numerical integration for frequency-limited Stein equations by interpolating the integrand at points on the unit circle. It also enables non-intrusive, data-driven balanced truncation and frequency-limited balanced truncation using experimentally measurable transfer function samples on the unit circle, without requiring access to a state-space realization. Numerical results for large-scale models demonstrate the effectiveness of the proposed methods as low-rank Stein equation solvers and data-driven model order reduction methods.
\end{abstract}

\end{frontmatter}

\section{Introduction}
Consider an $n$th-order discrete-time linear time-invariant (LTI) system $G(z)$ with the state-space realization
\begin{align}
G(z)=C(zE-A)^{-1}B,\nonumber
\end{align}
where $E\in\mathbb{R}^{n\times n}$, $A\in\mathbb{R}^{n\times n}$, $B\in\mathbb{R}^{n\times m}$, $C\in\mathbb{R}^{p\times n}$, and $z=e^{j\omega}$. Throughout this paper, $E$ is assumed to be invertible, and all eigenvalues of $E^{-1}A$ are assumed to lie inside the unit circle.

The controllability Gramian $P$ and observability Gramian $Q$ associated with the realization $(E,A,B,C)$ satisfy the Stein equations
\begin{align}
APA^\top - EPE^\top + BB^\top &= 0,\label{lyap_p}\\
A^\top Q A - E^\top Q E + C^\top C &= 0.\label{lyap_q}
\end{align}
When $n$ is large and $m\ll n$, $p\ll n$, the Gramians $P$ and $Q$ are often numerically low-rank, enabling their efficient computation in large-scale settings.

Balanced truncation (BT) \cite{moore1981principal} is one of the most widely used model order reduction (MOR) techniques and can produce accurate reduced-order models (ROMs). The balancing square-root algorithm (BSA) \cite{tombs1987truncated} is a numerically stable algorithm for BT. It requires the Gramians $P$ and $Q$ in square-root form, similar to a Cholesky factorization:
\[
P=Z_pZ_p^\top,\qquad Q=Z_qZ_q^\top.
\]
This motivates computing low-rank solutions of the Stein equations \eqref{lyap_p} and \eqref{lyap_q} in Cholesky-like factored form.
\subsection{Discrete-Time Pseudo-Optimal Rational Krylov Algorithm (DT-PORK)}
Define $V_{\mathrm{kry}}$ as
\begin{align}
V_{\mathrm{kry}}=
\begin{bmatrix}
(\sigma_1 E-A)^{-1}B & \cdots & (\sigma_k E-A)^{-1}B
\end{bmatrix}.
\end{align}
Also, define
\begin{align}
S_{\mathrm{kry}}=\mathrm{diag}(\sigma_1,\ldots,\sigma_k)\otimes I_m\quad \text{and}\quad L_{\mathrm{kry}}=\begin{bmatrix}1&\cdots&1\end{bmatrix}\otimes I_m.
\end{align}
Then, $V_{\mathrm{kry}}$ satisfies the Sylvester equation
\begin{align}
AV_{\mathrm{kry}}-EV_{\mathrm{kry}}S_{\mathrm{kry}}+BL_{\mathrm{kry}}=0.
\end{align}

Consider an $r=km$ order approximation of $G(z)$ given by
\begin{align}
\tilde{G}(z)=\tilde{C}(zI-\tilde{A})^{-1}\tilde{B},\nonumber
\end{align}
where $\tilde{A}\in\mathbb{R}^{r\times r}$, $\tilde{B}\in\mathbb{R}^{r\times m}$, and $\tilde{C}\in\mathbb{R}^{p\times r}$. Set
\begin{align}
\tilde{A}=S_{\mathrm{kry}}-\tilde{B}L_{\mathrm{kry}}, \qquad
\tilde{C}=CV_{\mathrm{kry}}.
\end{align}
Assuming that $(S_{\mathrm{kry}},L_{\mathrm{kry}})$ is observable, the following interpolation condition holds for any choice of $\tilde{B}$:
\begin{align}
G(\sigma_i)=\tilde{G}(\sigma_i), \qquad i=1,\ldots,k;
\end{align}cf. \cite{astolfi2010model,wolf2014h,ahmad2011krylov}.

Let $\tilde{\lambda}_i$ denote the poles of $\tilde{G}(z)$. The interpolation condition
\begin{align}
G\Big(\frac{1}{\overline{\tilde{\lambda}}_i}\Big)
=
\tilde{G}\Big(\frac{1}{\overline{\tilde{\lambda}}_i}\Big)
\label{opt_int}
\end{align}
is a subset of the optimality conditions for the $\mathcal{H}_2$-optimal MOR problem \cite{bunse2010h2}, namely,
\begin{align}
\frac{\partial}{\partial \tilde{C}}
\big(\|G(z)-\tilde{G}(z)\|_{\mathcal{H}_2}^2\big)=0.
\label{opt_cond}
\end{align}
In \cite{zulfiqar2025compression}, a choice of the free parameter $\tilde{B}$ is given for interpolation points $\sigma_i$ outside the unit circle. This choice places the poles of $\tilde{A}$ at $1/\overline{\sigma}_i$, thereby enforcing \eqref{opt_int} and \eqref{opt_cond}. Since $\sigma_i$ lie outside the unit circle, the poles of $\tilde{A}$ lie inside the unit circle, and $\tilde{G}(z)$ is stable.

Let $Q_s$ solve the Stein equation
\begin{align}
S_{\mathrm{kry}}^*Q_sS_{\mathrm{kry}}-Q_s-L_{\mathrm{kry}}^\top L_{\mathrm{kry}}=0.
\end{align}
It is shown in \cite{zulfiqar2025compression} that setting
\[
\tilde{B}=Q_s^{-1}S_{\mathrm{kry}}^{-*}L_{\mathrm{kry}}^\top
\]
enforces \eqref{opt_int} and \eqref{opt_cond}, and yields
\[
\tilde{A}
=
S_{\mathrm{kry}}-Q_s^{-1}S_{\mathrm{kry}}^{-*}L_{\mathrm{kry}}^\top L_{\mathrm{kry}}
=
Q_s^{-1}S_{\mathrm{kry}}^{-*}Q_s.
\]
Furthermore, $\tilde{P}=Q_s^{-1}$ is the controllability Gramian of $(\tilde{A},\tilde{B})$ and satisfies
\begin{align}
\tilde{A}\tilde{P}\tilde{A}^*-\tilde{P}+\tilde{B}\tilde{B}^*=0.
\end{align}
The Petrov--Galerkin approximation
\[
P\approx V_{\mathrm{kry}}Q_s^{-1}V_{\mathrm{kry}}^*
\]
converges monotonically to $P$ as the number of interpolation points $k$ increases.
\subsection{Factored ADI (fADI) Method for Sylvester Equations \cite{tu2009adi,benner2014computing}}
Consider the Sylvester equation
\begin{align}
AX_{\mathrm{sylv}}\hat{E}+EX_{\mathrm{sylv}}\hat{A}+B\hat{C}=0,
\label{sylv}
\end{align}
where $X_{\mathrm{sylv}}\in\mathbb{R}^{n\times v}$, $\hat{E}\in\mathbb{R}^{v\times v}$, $\hat{A}\in\mathbb{R}^{v\times v}$, and $\hat{C}\in\mathbb{R}^{m\times v}$.

Let $\{\alpha_i\}_{i=1}^{k}\subset\mathbb{C}$ and $\{\beta_i\}_{i=1}^{k}\subset\mathbb{C}$ be ADI shifts satisfying $\alpha_i\neq-\beta_i$. The fADI method \cite{tu2009adi,benner2014computing} defines
\begin{align}
v_i^{\mathrm{fadi}}&=(A+\alpha_i E)^{-1}B_{\perp,i-1}^{\mathrm{fadi}},\nonumber\\
w_i^{\mathrm{fadi}}&=(\hat{A}^\top+\overline{\beta}_i\hat{E}^\top)^{-1}
(C_{\perp,i-1}^{\mathrm{fadi}})^*,
\end{align}
where
\begin{align}
B_{\perp,i}^{\mathrm{fadi}}
&=B_{\perp,i-1}^{\mathrm{fadi}}
-(\alpha_i+\beta_i)Ev_i^{\mathrm{fadi}},\nonumber\\
C_{\perp,i}^{\mathrm{fadi}}
&=C_{\perp,i-1}^{\mathrm{fadi}}
-(\alpha_i+\beta_i)(w_i^{\mathrm{fadi}})^*\hat{E},
\end{align}
with $B_{\perp,0}^{\mathrm{fadi}}=B$ and
$C_{\perp,0}^{\mathrm{fadi}}=\hat{C}$.

Define
\begin{align}
V_{\mathrm{fadi}}^{(k)}
&=\begin{bmatrix}v_1^{\mathrm{fadi}}&\cdots&v_k^{\mathrm{fadi}}\end{bmatrix},\nonumber\\
W_{\mathrm{fadi}}^{(k)}
&=\begin{bmatrix}w_1^{\mathrm{fadi}}&\cdots&w_k^{\mathrm{fadi}}\end{bmatrix},\nonumber\\
X_{\mathrm{fadi}}^{(k)}
&=
\begin{bmatrix}
-(\alpha_1+\beta_1)&\cdots&0\\
\vdots&\ddots&\vdots\\
0&\cdots&-(\alpha_k+\beta_k)
\end{bmatrix}\otimes I_m.
\end{align}
Then, the fADI approximation
\[
X_{\mathrm{sylv}}\approx\tilde{X}_{\mathrm{sylv}}
=
V_{\mathrm{fadi}}^{(k)}
X_{\mathrm{fadi}}^{(k)}
(W_{\mathrm{fadi}}^{(k)})^*
\]
satisfies the residual relation
\[
A\tilde{X}_{\mathrm{sylv}}\hat{E}
+E\tilde{X}_{\mathrm{sylv}}\hat{A}
+B\hat{C}
=
B_{\perp,k}^{\mathrm{fadi}}C_{\perp,k}^{\mathrm{fadi}}.
\]

The low-rank solution of the Stein equation \eqref{lyap_p} can be obtained by setting
\[
\hat{E}=A^\top,\qquad
\hat{A}=-E^\top,\qquad
\hat{C}=B^\top.
\]
With these choices,
\begin{align}
w_i^{\mathrm{fadi}}
&=(-E+\overline{\beta}_iA)^{-1}
(C_{\perp,i-1}^{\mathrm{fadi}})^* \nonumber\\
&=\frac{1}{\overline{\beta}_i}
\left(A-\frac{1}{\overline{\beta}_i}E\right)^{-1}
(C_{\perp,i-1}^{\mathrm{fadi}})^*.
\end{align}
Since $C_{\perp,0}^{\mathrm{fadi}}=B^\top$, choosing
\[
\alpha_i=-\frac{1}{\overline{\beta}_i}
\]
gives
\[
w_1^{\mathrm{fadi}}=-\alpha_1v_1^{\mathrm{fadi}}.
\]

Define
\[
D_{\mathrm{fadi}}^{(k)}
=
\operatorname{diag}\left(
-\overline{\alpha}_1,\ldots,
-\overline{\alpha}_k\prod_{j=1}^{k-1}|\alpha_j|^2
\right)\otimes I_m.
\]
Then,
\[
W_{\mathrm{fadi}}^{(k)}
=
V_{\mathrm{fadi}}^{(k)}(D_{\mathrm{fadi}}^{(k)})^*,
\]
and the fADI approximation of $P$ becomes
\[
P\approx
V_{\mathrm{fadi}}^{(k)}
X_{\mathrm{fadi}}^{(k)}
D_{\mathrm{fadi}}^{(k)}
(V_{\mathrm{fadi}}^{(k)})^*.
\]
Moreover,
\[
X_{\mathrm{fadi}}^{(k)}D_{\mathrm{fadi}}^{(k)}
=
\operatorname{diag}\left(
|\alpha_1|^2-1,\ldots,
(|\alpha_k|^2-1)\prod_{j=1}^{k-1}|\alpha_j|^2
\right)\otimes I_m.
\]

If the shifts $\alpha_i$ lie outside the unit circle, i.e.,
$|\alpha_i|>1$, the approximation can be written as
\[
P\approx
\left(V_{\mathrm{fadi}}^{(k)}Z_{\mathrm{fadi}}^{(k)}\right)
\left(V_{\mathrm{fadi}}^{(k)}Z_{\mathrm{fadi}}^{(k)}\right)^*,
\]
where
\[
Z_{\mathrm{fadi}}^{(k)}
=
\begin{bmatrix}
\sqrt{|\alpha_1|^2-1}&\cdots&0\\
\vdots&\ddots&\vdots\\
0&\cdots&
\sqrt{(|\alpha_k|^2-1)\prod_{j=1}^{k-1}|\alpha_j|^2}
\end{bmatrix}\otimes I_m.
\]
This approach is known as the low-rank Cholesky-factor ADI (LRCF-ADI) method for Stein equations. A similar method is also presented in \cite{benner2011numerical}.
\section{Main Work}
This section first shows that, similar to the continuous-time LRCF-ADI method for Lyapunov equations \cite{benner2013efficient,benner2013reformulated}, the LRCF-ADI method for Stein equations \cite{benner2014computing} implicitly performs $\mathcal{H}_2$-pseudo-optimal MOR; see \cite{wolf2016adi} for the continuous-time case. Next, an effective shift-generation method is introduced, making the LRCF-ADI method for Stein equations fully automatic and eliminating the need for user intervention. A generalized LRCF-ADI method for Stein equations is then proposed, allowing the shifts $\alpha_i$ to be located anywhere in the complex plane, including on and inside the unit circle. Finally, applications in which shifts on the unit circle are more practical are discussed, for which the generalized LRCF-ADI method is more suitable.
\subsection{The LRCF-ADI Method for Stein Equations Implicitly Performs $\mathcal{H}_2$-Pseudo-Optimal MOR}
Define the ROM
\[
\tilde{G}^{(k)}(z)
=
\tilde{C}^{(k)}
\big(zI-\tilde{A}^{(k)}\big)^{-1}
\tilde{B}^{(k)}.
\]
In \cite{zulfiqar2026unified}, the interpolatory properties of fADI \cite{tu2009adi,benner2014computing} are studied, and it is shown that
\[
G(-\alpha_i)=\tilde{G}^{(k)}(-\alpha_i)
\]
when the reduced matrices are defined by
\begin{align}
\tilde A^{(k)}=
\begin{bmatrix}
\tilde A^{(k-1)} & 0\\
(\alpha_k+\beta_k)\mathbf{1}_{k-1}^\top & \beta_k
\end{bmatrix}\otimes I_m,\quad
\tilde B^{(k)}
=
\begin{bmatrix}
\tilde B^{(k-1)}\\
\alpha_k+\beta_k
\end{bmatrix}\otimes I_m,
\quad
\tilde C^{(k)}=CV_{\mathrm{fadi}}^{(k)},
\label{fadi_ss}
\end{align}
with $\tilde A^{(1)}=\beta_1$ and
$\tilde B^{(1)}=\alpha_1+\beta_1$, where $\mathbf{1}_{k}$ denotes the $k\times 1$ vector of ones.

The poles of $\tilde{G}^{(k)}(z)$ are given by $\beta_i$. Hence, fADI can be used to perform $\mathcal{H}_2$-pseudo-optimal MOR for discrete-time systems by choosing
\[
\alpha_i=-\sigma_i,
\qquad
\beta_i=-\frac{1}{\overline{\alpha}_i}.
\]
Thus, fADI provides a recursive implementation of DT-PORK \cite{zulfiqar2025compression} without modification.

Let $\tilde{P}^{(k)}$ denote the controllability Gramian associated with
$(\tilde A^{(k)},\tilde B^{(k)})$. By choosing
$\beta_i=-1/\overline{\alpha}_i$ with $|\alpha_i|>1$, the DT-PORK approximation of $P$ can be written as
\[
P\approx
V_{\mathrm{fadi}}^{(k)}
\tilde{P}^{(k)}
(V_{\mathrm{fadi}}^{(k)})^*.
\]
This is the same choice of shifts that reduces fADI to the LRCF-ADI method for Stein equations. The following proposition shows that
\[
\tilde{P}^{(k)}
=
X_{\mathrm{fadi}}^{(k)}D_{\mathrm{fadi}}^{(k)}.
\]
Therefore, DT-PORK and the LRCF-ADI method for Stein equations yield identical approximations of $P$ when $\alpha_i=-\sigma_i$.

\begin{proposition}
\label{lem:stein_solution}
Let $\tilde A^{(k)}$ and $\tilde B^{(k)}$ be defined as in \eqref{fadi_ss}. Let $\{\alpha_i\}_{i=1}^k$ and $\{\beta_i\}_{i=1}^k$ be fADI shifts satisfying $|\alpha_i|>1$, $\alpha_i+\beta_i\neq 0$, and $\beta_i=-1/\overline{\alpha}_i$. Define
\[
\rho_0:=1,\qquad
\rho_i:=\prod_{j=1}^i|\alpha_j|^2,\qquad
p_i:=(|\alpha_i|^2-1)\rho_{i-1}.
\]
Then
\[
\tilde P^{(k)}
=
\operatorname{diag}(p_1,\ldots,p_k)\otimes I_m
\]
solves the Stein equation
\begin{align}
\tilde A^{(k)}\tilde P^{(k)}(\tilde A^{(k)})^*
-\tilde P^{(k)}
+\tilde B^{(k)}(\tilde B^{(k)})^*
=0.
\label{proj_lyap}
\end{align}
\end{proposition}

\begin{proof}
Set
\[
\delta_i=\alpha_i+\beta_i,
\qquad
p_i=(|\alpha_i|^2-1)\prod_{j=1}^{i-1}|\alpha_j|^2.
\]
Since $\tilde A^{(k)}$, $\tilde B^{(k)}$, and $\tilde P^{(k)}$ are obtained by tensoring scalar matrices with $I_m$, it is sufficient to prove the result for the scalar case. Let $A_k$, $b_k$, and $P_k$ denote the corresponding scalar matrices. The entries of $A_k$ are
\[
(A_k)_{i\ell}=
\begin{cases}
\beta_i, & \ell=i,\\
\delta_i, & \ell<i,\\
0, & \ell>i,
\end{cases}
\]
with
\[
b_k=(\delta_1,\ldots,\delta_k)^\top,
\qquad
P_k=\operatorname{diag}(p_1,\ldots,p_k).
\]

Since $\beta_i=-1/\overline{\alpha}_i$, it follows that
\[
\beta_i|\alpha_i|^2=-\alpha_i,
\qquad
|\delta_i|^2+(|\beta_i|^2-1)(|\alpha_i|^2-1)=0.
\]
Furthermore,
\[
p_i=
\prod_{j=1}^{i}|\alpha_j|^2
-
\prod_{j=1}^{i-1}|\alpha_j|^2,
\]
and therefore
\begin{equation}
1+\sum_{\ell=1}^{i-1}p_\ell
=
\prod_{j=1}^{i-1}|\alpha_j|^2.
\label{eq:telescope}
\end{equation}

For $i>j$, the $(i,j)$ entry of $A_kP_kA_k^*+b_kb_k^*$ is
\[
\delta_i\left(
\overline{\delta_j}
\Bigl(1+\sum_{\ell=1}^{j-1}p_\ell\Bigr)
+p_j\overline{\beta_j}
\right).
\]
Using \eqref{eq:telescope}, the term in parentheses becomes
\[
\prod_{r=1}^{j-1}|\alpha_r|^2
\left(
\overline{\delta_j}
+
(|\alpha_j|^2-1)\overline{\beta_j}
\right)
=
\prod_{r=1}^{j-1}|\alpha_r|^2
\left(
\overline{\alpha_j}
+
|\alpha_j|^2\overline{\beta_j}
\right)
=0.
\]
Hence, all strictly lower-triangular entries vanish. By symmetry, the strictly upper-triangular entries also vanish.

For the diagonal entries,
\[
(A_kP_kA_k^*+b_kb_k^*)_{ii}
=
|\delta_i|^2
\Bigl(1+\sum_{\ell=1}^{i-1}p_\ell\Bigr)
+
|\beta_i|^2p_i.
\]
Using \eqref{eq:telescope}, this becomes
\[
|\delta_i|^2\prod_{j=1}^{i-1}|\alpha_j|^2
+
|\beta_i|^2p_i
=
p_i+
\prod_{j=1}^{i-1}|\alpha_j|^2
\left(
|\delta_i|^2
+
(|\beta_i|^2-1)(|\alpha_i|^2-1)
\right)
=
p_i.
\]
Thus,
\[
A_kP_kA_k^*-P_k+b_kb_k^*=0.
\]
Tensoring with $I_m$ yields
\[
\tilde A^{(k)}\tilde P^{(k)}(\tilde A^{(k)})^*
-\tilde P^{(k)}
+\tilde B^{(k)}(\tilde B^{(k)})^*
=0.
\]
\end{proof}

It follows that the LRCF-ADI method for Stein equations \cite{benner2014computing,benner2011numerical} is a recursive implementation of DT-PORK \cite{zulfiqar2025compression}. In the LRCF-ADI method, the matrices $\tilde A^{(k)}$, $\tilde B^{(k)}$, and $\tilde C^{(k)}$ are not collected because a ROM is not required to approximate $P$. However, collecting these matrices yields an $\mathcal{H}_2$-pseudo-optimal ROM.
\subsection{Automatic Shift Generation for the LRCF-ADI Method for Stein Equations\label{sub1}}

We now consider the error of the ROM obtained from fADI with
\[
\alpha_i=-\sigma_i,
\qquad
\beta_i=-\frac{1}{\overline{\alpha}_i}.
\]
It is shown in \cite{zulfiqar2025compression} that the resulting ROM satisfies
\begin{align}
\|G(z)-\tilde{G}^{(k)}(z)\|_{\mathcal{H}_2}^2
&=\|G(z)\|_{\mathcal{H}_2}^2-\|\tilde{G}^{(k)}(z)\|_{\mathcal{H}_2}^2,\nonumber\\
&=\mathrm{trace}(CPC^\top)
-\mathrm{trace}\big(\tilde C^{(k)}\tilde P^{(k)}(\tilde C^{(k)})^*\big),\nonumber\\
&=\mathrm{trace}\Big(
C\big(P-V_{\mathrm{fadi}}^{(k)}
\tilde P^{(k)}
(V_{\mathrm{fadi}}^{(k)})^*\big)C^\top
\Big).
\end{align}
Furthermore,
\begin{align}
\|G(z)-\tilde{G}^{(k)}(z)\|_{\mathcal{H}_2}^2
\leq
\|G(z)-\tilde{G}^{(k-1)}(z)\|_{\mathcal{H}_2}^2.
\label{mono_decay}
\end{align}
Therefore, $\tilde{G}^{(k)}(z)$ and
$V_{\mathrm{fadi}}^{(k)}\tilde P^{(k)}(V_{\mathrm{fadi}}^{(k)})^*$
converge monotonically to $G(z)$ and $P$, respectively, as the number of shifts increases.

Define
\[
G_{\perp}^{(k)}(z)
=
C(zE-A)^{-1}B_{\perp,k}^{\mathrm{fadi}}.
\]
The approximation error can then be factorized as
\[
G(z)-\tilde{G}^{(k)}(z)
=
G_{\perp}^{(k)}(z)
\left(
(\mathbf{1}_{k}^\top\otimes I_m)
\big(zI-\tilde A^{(k)}\big)^{-1}
\tilde B^{(k)}
+I_m
\right);
\]
see \cite{wolf2014h}.

It is shown in \cite{zulfiqar2026unified} that fADI is a special case of the recursive interpolation framework in \cite{wolf2014h,panzer2014model}. Hence, the error analysis of this framework also applies to fADI. The aim of recursive interpolation is to make the pair
$(E^{-1}A,E^{-1}B_{\perp,k}^{\mathrm{fadi}})$ uncontrollable as the number of interpolation points increases. The transfer function $G_{\perp}^{(k)}(z)$ can be viewed as a deflated version of $G(z)$. In particular, frequency-response peaks associated with controllable poles that have been captured by $\tilde{G}^{(k)}(z)$ are attenuated in $G_{\perp}^{(k)}(z)$ \cite{zulfiqar2026unified}. When these peaks are sufficiently attenuated, the error
$G(z)-\tilde{G}^{(k)}(z)$ decreases significantly. If no controllable poles remain in $G_{\perp}^{(k)}(z)$, then $\tilde{G}^{(k)}(z)$ recovers $G(z)$ exactly, and
\[
G(z)-\tilde{G}^{(k)}(z)=0;
\]
see \cite{wolf2014h}.

The monotonicity property in \eqref{mono_decay} does not indicate the rate of error decay. It is well known that the $\mathcal{H}_2$ norm error is dominated by the most controllable and observable poles of $G(z)$. If $\tilde{G}^{(k)}(z)$ interpolates at the reciprocals of these poles, the error decreases rapidly \cite{bunse2010h2,gugercin2008h_2}. Although $G(z)$ and $G_{\perp}^{(k)}(z)$ have the same poles, the poles already captured by $\tilde{G}^{(k)}(z)$ become weakly controllable or uncontrollable in $G_{\perp}^{(k)}(z)$. Therefore, the reciprocals of the dominant poles of $G_{\perp}^{(k)}(z)$ can be used as interpolation points $\sigma_i=-\alpha_i$ for the LRCF-ADI method.

In large-scale settings, computing the poles of $G_{\perp}^{(k)}(z)$ is infeasible. Instead, Ritz values of $E^{-1}A$ can be computed by projection onto the interpolation basis accumulated during the LRCF-ADI iteration. Define
\[
V_{\mathrm{proj}}
=
\mathrm{orth}\left(
\begin{bmatrix}
v_1^{\mathrm{fadi}} & \cdots & v_i^{\mathrm{fadi}}
\end{bmatrix}
\right).
\]
An implicit restart can be used: if the number of columns of
$V_{\mathrm{proj}}$ exceeds a prescribed limit, the previous vectors are discarded and a new basis is accumulated. Such restart mechanisms are commonly used in eigenvalue algorithms to limit the basis dimension \cite{saad2011numerical}.

Next, form
\[
E_p=V_{\mathrm{proj}}^\top E V_{\mathrm{proj}},\qquad
A_p=V_{\mathrm{proj}}^\top A V_{\mathrm{proj}},\qquad
B_p=V_{\mathrm{proj}}^\top B_{\perp,i}^{\mathrm{fadi}}.
\]
Compute the eigendecomposition
\[
E_p^{-1}A_p
=
T_p\,\mathrm{diag}(\hat{\lambda}_1,\ldots,\hat{\lambda}_i)T_p^{-1},
\]
and define
\[
r_{p,l}=T_p^{-1}(l,:)B_p.
\]
The most controllable Ritz value is selected as the value
$\hat{\lambda}_d$ corresponding to the largest quantity
\[
\phi_l
=
\frac{\|r_{p,l}\|_2^2}
{1-|\hat{\lambda}_l|^2};
\]
see \cite{rommes2007methods}. The next shift is then chosen as
\[
\alpha_i=-\frac{1}{\hat{\lambda}_d}.
\]
If all Ritz values of $E_p^{-1}A_p$ lie outside the unit circle, the previous shift is repeated. Thus, after initialization with an arbitrary shift, the LRCF-ADI method generates subsequent shifts automatically without user intervention.
\subsection{A Generalized LRCF-ADI Algorithm for Stein Equations}

In \cite{zulfiqar2025compression}, it is shown that DT-PORK-based, and hence LRCF-ADI-based, BT can be performed non-intrusively using samples $G(\sigma_i)$ without access to the state-space realization $(E,A,B,C)$. However, DT-PORK requires $|\sigma_i|>1$, whereas experimentally measurable frequency-response samples are available only on the unit circle, i.e., $|\sigma_i|=1$. Therefore, the non-intrusive DT-PORK-based BT approach in \cite{zulfiqar2025compression} is not data-driven.

The projected Stein equation \eqref{proj_lyap} has a unique solution provided that $|\beta_i|<1$ and $\alpha_i\neq-\beta_i$. Hence, if the $\mathcal{H}_2$-pseudo-optimality condition is not imposed, the shifts $\alpha_i$ may be selected anywhere in the complex plane, subject to $\alpha_i\neq-\beta_i$. The following theorem shows that $\tilde{P}^{(k)}$ can still be factorized as
\[
\tilde{P}^{(k)}=\tilde{Z}^{(k)}(\tilde{Z}^{(k)})^*,
\]
which leads to a generalized LRCF-ADI algorithm for Stein equations.

\begin{theorem}
\label{th1}
Let $\tilde A^{(k)}$ and $\tilde B^{(k)}$ be defined as in \eqref{fadi_ss}. Let $\{\alpha_i\}_{i=1}^k$ and $\{\beta_i\}_{i=1}^k$ be fADI shifts satisfying $\alpha_i+\beta_i\neq 0$, $|\beta_i|<1$, and $\beta_i\neq 0$. Define the upper-triangular matrix
\[
\hat{T}^{(k)}
=
\begin{bmatrix}
\hat{T}^{(k-1)} & t_1^{(k)}\\
0 & t_2^{(k)}
\end{bmatrix}\otimes I_m,
\]
where
\begin{align}
\left[t_1^{(k)}\right]_i
&=
\frac{
1+\displaystyle\sum_{r=i}^{k-1}
\overline{(\alpha_r+\beta_r)}\,\hat{T}^{(k-1)}_{ir}
-\displaystyle\sum_{r=1}^{i-1}
\left(\frac{1}{\beta_r}-\overline{\beta_r}\right)
\left[t_1^{(k)}\right]_r
}{
\frac{1}{\beta_i}-\overline{\beta_k}
},
\nonumber\\
t_2^{(k)}
&=
\frac{
1-\displaystyle\sum_{r=1}^{k-1}
\left(\frac{1}{\beta_r}-\overline{\beta_r}\right)
\left[t_1^{(k)}\right]_r
}{
\frac{1}{\beta_k}-\overline{\beta_k}
},
\label{compute_t}
\end{align}
for $i=1,\ldots,k-1$.

Furthermore, define
\[
Z^{(k)}
=
\begin{bmatrix}
\sqrt{\frac{1-|\beta_1|^2}{|\beta_1|^2}} & \cdots & 0 \\[1ex]
\vdots & \ddots & \vdots \\[1ex]
0 & \cdots &
\sqrt{\frac{1-|\beta_k|^2}{\prod_{j=1}^{k}|\beta_j|^2}}
\end{bmatrix}
\otimes I_m
\]
and
\[
\tilde{Z}^{(k)}
=
X_{\mathrm{fadi}}^{(k)}(\hat{T}^{(k)})^*Z^{(k)}.
\]
Then, the controllability Gramian $\tilde{P}^{(k)}$ admits the factorization
\[
\tilde{P}^{(k)}
=
\tilde{Z}^{(k)}(\tilde{Z}^{(k)})^*.
\]
\end{theorem}

\begin{proof}
The proof is given in the appendix.
\end{proof}

By Theorem \ref{th1}, a projection-based approximation of $P$ is given by
\[
P\approx
\big(V_{\mathrm{fadi}}^{(k)}\tilde{Z}^{(k)}\big)
\big(V_{\mathrm{fadi}}^{(k)}\tilde{Z}^{(k)}\big)^*.
\]
The following proposition gives the residual expression for this approximation.

\begin{proposition}
Let the conditions of Theorem \ref{th1} hold. Define
\[
\hat{P}
=
V_{\mathrm{fadi}}^{(k)}
\tilde{P}^{(k)}
(V_{\mathrm{fadi}}^{(k)})^*
\]
and
\[
L^{(k)}=-\mathbf{1}_{k}^\top\otimes I_m.
\]
Then,
\begin{align}
A\hat{P}A^\top-E\hat{P}E^\top+BB^\top
=
F^{(k)}M^{(k)}(F^{(k)})^*,
\end{align}
where
\begin{align}
F^{(k)}
&=
\begin{bmatrix}
B_{\perp,k}^{\mathrm{fadi}} &
EV_{\mathrm{fadi}}^{(k)}
\big(
\tilde{B}^{(k)}
-\tilde{A}^{(k)}
\tilde{P}^{(k)}
(L^{(k)})^\top
\big)
\end{bmatrix},
\nonumber\\
M^{(k)}
&=
\begin{bmatrix}
I_m+L^{(k)}\tilde{P}^{(k)}(L^{(k)})^\top & I_m\\
I_m & \mathbf{0}
\end{bmatrix}.
\end{align}
\end{proposition}

\begin{proof}
The proof uses the following relations from \cite{zulfiqar2026unified}:
\begin{align}
AV_{\mathrm{fadi}}^{(k)}
-EV_{\mathrm{fadi}}^{(k)}\tilde{A}^{(k)}
+B_{\perp,k}^{\mathrm{fadi}}L^{(k)}
&=0,
\nonumber\\
B_{\perp,k}^{\mathrm{fadi}}
&=
B-EV_{\mathrm{fadi}}^{(k)}\tilde{B}^{(k)}.
\end{align}
Moreover,
\begin{align}
BB^\top
&=
B_{\perp,k}^{\mathrm{fadi}}(B_{\perp,k}^{\mathrm{fadi}})^*
+
B_{\perp,k}^{\mathrm{fadi}}(\tilde{B}^{(k)})^*
(V_{\mathrm{fadi}}^{(k)})^*E^\top+
EV_{\mathrm{fadi}}^{(k)}\tilde{B}^{(k)}
(B_{\perp,k}^{\mathrm{fadi}})^*
+
EV_{\mathrm{fadi}}^{(k)}\tilde{B}^{(k)}
(\tilde{B}^{(k)})^*
(V_{\mathrm{fadi}}^{(k)})^*E^\top,
\nonumber\\
AV_{\mathrm{fadi}}^{(k)}
&=
EV_{\mathrm{fadi}}^{(k)}\tilde{A}^{(k)}
-B_{\perp,k}^{\mathrm{fadi}}L^{(k)},
\nonumber\\
\tilde B^{(k)}(\tilde B^{(k)})^*
&=
-\tilde A^{(k)}\tilde P^{(k)}(\tilde A^{(k)})^*
+\tilde P^{(k)}.
\end{align}
Therefore,
\begin{align}
&A\hat{P}A^\top-E\hat{P}E^\top+BB^\top
\nonumber\\
&=
AV_{\mathrm{fadi}}^{(k)}
\tilde P^{(k)}
(V_{\mathrm{fadi}}^{(k)})^*A^\top
-EV_{\mathrm{fadi}}^{(k)}
\tilde P^{(k)}
(V_{\mathrm{fadi}}^{(k)})^*E^\top
+BB^\top
\nonumber\\
&=
\big(
EV_{\mathrm{fadi}}^{(k)}\tilde A^{(k)}
-B_{\perp,k}^{\mathrm{fadi}}L^{(k)}
\big)
\tilde P^{(k)}
\big(
EV_{\mathrm{fadi}}^{(k)}\tilde A^{(k)}
-B_{\perp,k}^{\mathrm{fadi}}L^{(k)}
\big)^*-EV_{\mathrm{fadi}}^{(k)}
\tilde P^{(k)}
(V_{\mathrm{fadi}}^{(k)})^*E^\top
+BB^\top
\nonumber\\
&=
EV_{\mathrm{fadi}}^{(k)}
\Big(
\tilde A^{(k)}
\tilde P^{(k)}
(\tilde A^{(k)})^*
-\tilde P^{(k)}
\Big)
(V_{\mathrm{fadi}}^{(k)})^*E^\top-EV_{\mathrm{fadi}}^{(k)}
\tilde A^{(k)}
\tilde P^{(k)}
(L^{(k)})^\top
(B_{\perp,k}^{\mathrm{fadi}})^*
\nonumber\\
&\quad
-B_{\perp,k}^{\mathrm{fadi}}
L^{(k)}
\tilde P^{(k)}
(\tilde A^{(k)})^*
(V_{\mathrm{fadi}}^{(k)})^*E^\top+B_{\perp,k}^{\mathrm{fadi}}
L^{(k)}
\tilde P^{(k)}
(L^{(k)})^\top
(B_{\perp,k}^{\mathrm{fadi}})^*
+BB^\top
\nonumber\\
&=
B_{\perp,k}^{\mathrm{fadi}}
(B_{\perp,k}^{\mathrm{fadi}})^*
+
B_{\perp,k}^{\mathrm{fadi}}
(\tilde B^{(k)})^*
(V_{\mathrm{fadi}}^{(k)})^*E^\top+EV_{\mathrm{fadi}}^{(k)}
\tilde B^{(k)}
(B_{\perp,k}^{\mathrm{fadi}})^*
-EV_{\mathrm{fadi}}^{(k)}
\tilde A^{(k)}
\tilde P^{(k)}
(L^{(k)})^\top
(B_{\perp,k}^{\mathrm{fadi}})^*
\nonumber\\
&\quad
-B_{\perp,k}^{\mathrm{fadi}}
L^{(k)}
\tilde P^{(k)}
(\tilde A^{(k)})^*
(V_{\mathrm{fadi}}^{(k)})^*E^\top+B_{\perp,k}^{\mathrm{fadi}}
L^{(k)}
\tilde P^{(k)}
(L^{(k)})^\top
(B_{\perp,k}^{\mathrm{fadi}})^*
\nonumber\\
&=
B_{\perp,k}^{\mathrm{fadi}}
\Big(
I_m+L^{(k)}\tilde P^{(k)}(L^{(k)})^\top
\Big)
(B_{\perp,k}^{\mathrm{fadi}})^*+EV_{\mathrm{fadi}}^{(k)}
\Big(
\tilde B^{(k)}
-\tilde A^{(k)}
\tilde P^{(k)}
(L^{(k)})^\top
\Big)
(B_{\perp,k}^{\mathrm{fadi}})^*
\nonumber\\
&\quad
+B_{\perp,k}^{\mathrm{fadi}}
\Big(
(\tilde B^{(k)})^*
-L^{(k)}
\tilde P^{(k)}
(\tilde A^{(k)})^*
\Big)
(V_{\mathrm{fadi}}^{(k)})^*E^\top
\nonumber\\
&=
F^{(k)}M^{(k)}(F^{(k)})^*.\nonumber
\end{align}
\end{proof}

The residual $F^{(k)}M^{(k)}(F^{(k)})^*$ has rank at most $2m$. Hence, its Frobenius norm or spectral norm can be computed efficiently when $m\ll n$.

In practical applications, a real-valued approximation is often preferred:
\[
P\approx
\big(V_r^{(k)}\tilde{Z}_r^{(k)}\big)
\big(V_r^{(k)}\tilde{Z}_r^{(k)}\big)^\top,
\]
where $V_r^{(k)}$ and $\tilde{Z}_r^{(k)}$ are real-valued. This also reduces storage requirements. In the following, the realification strategy for fADI in \cite{benner2014computing} is extended to the present setting.

Assume that the shifts $\alpha_i$ and $\beta_i$ are ordered according to one of the following cases:
\begin{enumerate}
\item $\alpha_i\in\mathbb{R}$ and $\beta_i\in\mathbb{R}$.
\item $\alpha_i\in\mathbb{C}$, $\alpha_{i+1}=\overline{\alpha}_i$, $\beta_i\in\mathbb{C}$, and $\beta_{i+1}=\overline{\beta}_i$.
\item $\alpha_i\in\mathbb{C}$, $\alpha_{i+1}=\overline{\alpha}_i$, $\beta_i\in\mathbb{R}$, and $\beta_{i+1}\in\mathbb{R}$.
\item $\alpha_i\in\mathbb{R}$, $\alpha_{i+1}\in\mathbb{R}$, $\beta_i\in\mathbb{C}$, and $\beta_{i+1}=\overline{\beta}_i$.
\end{enumerate}

Define $t(\mu_i,\mu_{i+1},\nu_i,\nu_{i+1})$ by
\begin{align}
t(\mu_i,\mu_{i+1},\nu_i,\nu_{i+1})
=
\begin{cases}
I_m,
&
\text{if } \mu_i\in\mathbb{R},\;
\nu_i\in\mathbb{R},
\\[2.5ex]
\dfrac{1}{\overline{\mu}_i+\nu_i}
\begin{bmatrix}
\operatorname{Re}(\mu_i)+\nu_i & -\operatorname{Im}(\mu_i)\\
-j\,\operatorname{Im}(\mu_i) & \operatorname{Im}(\mu_i)
\end{bmatrix}\otimes I_m,
&
\begin{array}{l}
\text{if } \mu_i\in\mathbb{C},\;
\mu_{i+1}=\overline{\mu}_i,\\
\nu_i\in\mathbb{C},\;
\nu_{i+1}=\overline{\nu}_i,
\end{array}
\\[5ex]
\dfrac{1}{\overline{\mu}_i+\nu_i}
\begin{bmatrix}
\operatorname{Re}(\mu_i)+\nu_i & -\operatorname{Im}(\mu_i)\\
-j\,\operatorname{Im}(\mu_i) & \operatorname{Im}(\mu_i)
\end{bmatrix}\otimes I_m,
&
\begin{array}{l}
\text{if } \mu_i\in\mathbb{C},\;
\mu_{i+1}=\overline{\mu}_i,\\
\nu_i\in\mathbb{R},\;
\nu_{i+1}\in\mathbb{R},
\end{array}
\\[5ex]
\begin{bmatrix}
1 & \dfrac{1}{\mu_{i+1}+\nu_i}\\
0 & -\dfrac{1}{\mu_{i+1}+\nu_i}
\end{bmatrix}\otimes I_m,
&
\begin{array}{l}
\text{if } \mu_i,\mu_{i+1}\in\mathbb{R},\\
\nu_i\in\mathbb{C},\;
\nu_{i+1}=\overline{\nu}_i.
\end{array}
\end{cases}
\label{gen_t}
\end{align}

Define
\begin{align}
T_v^{(k)}
&=
\mathrm{blkdiag}\big(T_v^{(k-1)},t_v^{(k)}\big),
\nonumber\\
T_w^{(k)}
&=
\mathrm{blkdiag}\big(T_w^{(k-1)},t_w^{(k)}\big),
\nonumber\\
T_p^{(k)}
&=
\mathrm{blkdiag}\big(T_p^{(k-1)},t_p^{(k)}\big),
\end{align}
where $t_v^{(k)}$, $t_w^{(k)}$, and $t_p^{(k)}$ are obtained from \eqref{gen_t} by setting
\[
(\mu_i,\nu_i)=(\alpha_i,\beta_i),\qquad
(\mu_i,\nu_i)=(\overline{\beta}_i,\overline{\alpha}_i),\qquad
(\mu_i,\nu_i)=\left(-\frac{1}{\overline{\beta}_i},\beta_i\right),
\]
respectively.

The approximation $\hat{P}$ can then be realified as
\begin{align}
\hat{P}
&=
\underbrace{V_{\mathrm{fadi}}^{(k)}T_v^{(k)}}_{V_r^{(k)}}
\underbrace{
(T_v^{(k)})^{-1}
X_{\mathrm{fadi}}^{(k)}
(T_w^{(k)})^{-*}
}_{X_r^{(k)}}
\underbrace{
(T_w^{(k)})^*
(\hat{T}^{(k)})^*
T_p^{(k)}
}_{(T_r^{(k)})^\top}
\nonumber\\
&\quad\times
\underbrace{
(T_p^{(k)})^{-1}
Z^{(k)}(Z^{(k)})^*
(T_p^{(k)})^{-*}
}_{Z_r^{(k)}(Z_r^{(k)})^\top}
\underbrace{
(T_p^{(k)})^*
\hat{T}^{(k)}
T_w^{(k)}
}_{T_r^{(k)}}
\underbrace{
(T_w^{(k)})^{-1}
(X_{\mathrm{fadi}}^{(k)})^*
(T_v^{(k)})^{-*}
}_{(X_r^{(k)})^\top}
\underbrace{
(T_v^{(k)})^*
(V_{\mathrm{fadi}}^{(k)})^*
}_{(V_r^{(k)})^\top},
\end{align}
where $V_r^{(k)}$, $X_r^{(k)}$, $T_r^{(k)}$, and $Z_r^{(k)}$ are real-valued.

Represent the realified matrices recursively as
\begin{align}
V_r^{(k)}
&=
\begin{bmatrix}
V_r^{(k-1)} & v_r^{(k)}
\end{bmatrix},
\nonumber\\
X_r^{(k)}
&=
\mathrm{blkdiag}\big(X_r^{(k-1)},x_r^{(k)}\big),
\nonumber\\
Z_r^{(k)}
&=
\mathrm{blkdiag}\big(Z_r^{(k-1)},z_r^{(k)}\big).
\end{align}
If $\alpha_i,\beta_i\in\mathbb{R}$, then
\begin{align}
v_r^{(i)}=v_i^{\mathrm{fadi}},
\qquad
x_r^{(i)}=-(\alpha_i+\beta_i)I_m.
\end{align}
If $\alpha_i\in\mathbb{C}$, $\alpha_{i+1}=\overline{\alpha}_i$,
$\beta_i\in\mathbb{C}$, and $\beta_{i+1}=\overline{\beta}_i$, then
\begin{align}
v_r^{(i)}
&=
\begin{bmatrix}
v_i^{\mathrm{fadi}} & v_{i+1}^{\mathrm{fadi}}
\end{bmatrix}
t_v^{(i)}
=
\begin{bmatrix}
\operatorname{Re}(v_i^{\mathrm{fadi}}) &
\operatorname{Im}(v_i^{\mathrm{fadi}})
\end{bmatrix},
\nonumber\\
x_r^{(i)}
&=
(t_v^{(i)})^{-1}
\begin{bmatrix}
-(\alpha_i+\beta_i)I_m & \mathbf{0}\\
\mathbf{0} & -(\alpha_{i+1}+\beta_{i+1})I_m
\end{bmatrix}
(t_w^{(i)})^{-*}
\nonumber\\
&=
\begin{bmatrix}
2\big(\operatorname{Re}(-\alpha_i)-\operatorname{Re}(\beta_i)\big)
&
\frac{
|\alpha_i+\beta_i|^2
+2\operatorname{Im}(\beta_i)
\big(\operatorname{Im}(-\alpha_i)-\operatorname{Im}(\beta_i)\big)
}{
\operatorname{Im}(\beta_i)
}
\\[3ex]
\frac{
|\alpha_i+\beta_i|^2
-2\operatorname{Im}(-\alpha_i)
\big(\operatorname{Im}(-\alpha_i)-\operatorname{Im}(\beta_i)\big)
}{
\operatorname{Im}(-\alpha_i)
}
&
\frac{
\big(\operatorname{Re}(-\alpha_i)-\operatorname{Re}(\beta_i)\big)
\big(
2\operatorname{Im}(\beta_i)\operatorname{Im}(-\alpha_i)
+|\alpha_i+\beta_i|^2
\big)
}{
\operatorname{Im}(\beta_i)\operatorname{Im}(-\alpha_i)
}
\end{bmatrix}
\otimes I_m.
\label{x_case2}
\end{align}

If $\alpha_i\in\mathbb{C}$, $\alpha_{i+1}=\overline{\alpha}_i$,
and $\beta_i,\beta_{i+1}\in\mathbb{R}$, then
\begin{align}
v_r^{(i)}
&=
\begin{bmatrix}
v_i^{\mathrm{fadi}} & v_{i+1}^{\mathrm{fadi}}
\end{bmatrix}
t_v^{(i)}
=
\begin{bmatrix}
\operatorname{Re}(v_i^{\mathrm{fadi}}) &
\operatorname{Im}(v_i^{\mathrm{fadi}})
\end{bmatrix},
\nonumber\\
x_r^{(i)}
&=
(t_v^{(i)})^{-1}
\begin{bmatrix}
-(\alpha_i+\beta_i)I_m & \mathbf{0}\\
\mathbf{0} & -(\alpha_{i+1}+\beta_{i+1})I_m
\end{bmatrix}
(t_w^{(i)})^{-*}
\nonumber\\
&=
\begin{bmatrix}
\operatorname{Re}(-\alpha_i)-\beta_i
+\operatorname{Re}(-\alpha_{i+1})-\beta_{i+1}
&
\beta_{i+1}^2
-2\operatorname{Re}(-\alpha_i)\beta_{i+1}
+|-\alpha_i|^2
\\[2ex]
\frac{
\big(\operatorname{Re}(-\alpha_i)-\beta_i\big)
\big(\operatorname{Re}(-\alpha_{i+1})-\beta_{i+1}\big)
-\operatorname{Im}(-\alpha_i)^2
}{
\operatorname{Im}(-\alpha_i)
}
&
-\frac{
\big(\operatorname{Re}(-\alpha_i)-\beta_i\big)
\big(
\beta_{i+1}^2
-2\operatorname{Re}(-\alpha_i)\beta_{i+1}
+|-\alpha_i|^2
\big)
}{
\operatorname{Im}(-\alpha_i)
}
\end{bmatrix}
\otimes I_m.
\label{x_case3}
\end{align}

If $\alpha_i,\alpha_{i+1}\in\mathbb{R}$, $\beta_i\in\mathbb{C}$,
and $\beta_{i+1}=\overline{\beta}_i$, then
\begin{align}
v_r^{(i)}
&=
\begin{bmatrix}
v_i^{\mathrm{fadi}} & v_{i+1}^{\mathrm{fadi}}
\end{bmatrix}
t_v^{(i)}
=
\begin{bmatrix}
v_i^{\mathrm{fadi}} &
(A+\alpha_{i+1}E)^{-1}Ev_i^{\mathrm{fadi}}
\end{bmatrix},
\nonumber\\
x_r^{(i)}
&=
(t_v^{(i)})^{-1}
\begin{bmatrix}
-(\alpha_i+\beta_i)I_m & \mathbf{0}\\
\mathbf{0} & -(\alpha_{i+1}+\beta_{i+1})I_m
\end{bmatrix}
(t_w^{(i)})^{-*}
\nonumber\\
&=
\begin{bmatrix}
-\alpha_i-\operatorname{Re}(\beta_i)
-\alpha_{i+1}-\operatorname{Re}(\beta_i)
&
\frac{
\big(-\alpha_i-\operatorname{Re}(\beta_i)\big)
\big(-\alpha_{i+1}-\operatorname{Re}(\beta_i)\big)
-\operatorname{Im}(\beta_i)^2
}{
\operatorname{Im}(\beta_i)
}
\\[3ex]
\alpha_{i+1}^2
+2\operatorname{Re}(\beta_i)\alpha_{i+1}
+|\beta_i|^2
&
\frac{
\big(-\alpha_i-\operatorname{Re}(\beta_i)\big)
\big(
\alpha_{i+1}^2
+2\operatorname{Re}(\beta_i)\alpha_{i+1}
+|\beta_i|^2
\big)
}{
\operatorname{Im}(\beta_i)
}
\end{bmatrix}
\otimes I_m.
\label{x_case4}
\end{align}

The factors $z_r^{(i)}$ can be obtained from
\[
z_r^{(i)}(z_r^{(i)})^\top
=
(t_p^{(i)})^{-1}
\begin{bmatrix}
\dfrac{1-|\beta_i|^2}{\prod_{j=1}^{i}|\beta_j|^2}I_m & \mathbf{0}\\
\mathbf{0} &
\dfrac{1-|\beta_{i+1}|^2}{\prod_{j=1}^{i+1}|\beta_j|^2}I_m
\end{bmatrix}
(t_p^{(i)})^{-*}.
\]
Define
\[
c_i=
\begin{cases}
1, & \beta_i\in\mathbb{R},\\
2, & \beta_i\notin\mathbb{R},
\end{cases}
\qquad
\rho_i=|\beta_i|^{-2},
\]
and
\[
\Gamma_1=1,
\qquad
\Gamma_i=\prod_{\ell=1}^{i-1}\rho_\ell^{c_\ell},
\qquad
i=2,\ldots,k.
\]
Then,
\begin{align}
z_r^{(i)}
=
\begin{cases}
\sqrt{\Gamma_i(\rho_i-1)}I_m,
& \beta_i\in\mathbb{R},
\\[4ex]
\sqrt{\Gamma_i(\rho_i-1)}
\begin{bmatrix}
\sqrt{\rho_i+1} & 0
\\[2ex]
\displaystyle
\frac{\operatorname{Re}(\beta_i)}{\operatorname{Im}(\beta_i)}
\frac{\rho_i-1}{\sqrt{\rho_i+1}}
&
\displaystyle
\frac{|1-\beta_i^2|}{|\operatorname{Im}(\beta_i)|}
\frac{1}{\sqrt{\rho_i+1}}
\end{bmatrix}
\otimes I_m,
& \beta_i\notin\mathbb{R}.
\end{cases}
\label{compute_z}
\end{align}

We are now ready to present the proposed ``Generalized LRCF-ADI Method for Stein Equations (G-LRCF-ADI)''.

\noindent\rule{\textwidth}{0.8pt}

\textbf{Algorithm: G-LRCF-ADI}

\noindent\rule{\textwidth}{0.8pt}

\textbf{Input:} Matrices of the Stein equation \eqref{lyap_p}: $(E,A,B)$; ADI shifts: $\{\alpha_i\}_{i=1}^{k}\in\mathbb{C}$; Desired poles: $\{\beta_i\}_{i=1}^{k}$ satisfying $\beta_i\neq-\alpha_i$, $|\beta_i|<1$, and $\beta_i\neq 0$; Tolerance: $1>\tau>0$.

\textbf{Output:} Approximation: $P\approx \big(V_r^{(k)}\tilde{Z}_r^{(k)}\big)\big(V_r^{(k)}\tilde{Z}_r^{(k)}\big)^\top$.

1. \textbf{Initialize:} $i=1$, $V_r^{(0)}=[\;]$, $\tilde{Z}_r^{(0)}=[\;]$, $Z_r^{(0)}=[\;]$; $T^{(0)}=[\;]$, $T_p^{(0)}=[\;]$, $L^{(0)}=[\;]$, $B_{\perp,0}^{\mathrm{fadi}}=B$, $F^{(0)}=B$, $M=I$, $\Gamma_1=1$, $A_r^{(0)}=[\;]$, $B_r^{(0)}=[\;]$.

2. \textcolor{blue}{\textbf{while} $\frac{\|F^{(i)}M^{(i)}(F^{(i)})^\top\|}{\|BB^\top\|}>\tau$ \textbf{do}}

3. Solve $\big(A + \alpha_i E)v_i^{\mathrm{fadi}}= B_{\perp,i-1}^{\mathrm{fadi}}$ for $v_i^{\mathrm{fadi}}$ and set $\rho_i = |\beta_i|^{-2}$. 

4. \textcolor{blue}{\textbf{If} $\alpha_i\in\mathbb{R}$ and $\beta_i\in\mathbb{R}$ \textbf{do}}

5. Set $x_r^{(i)}=-(\alpha_i+\beta_i)I_m$, $z_r^{(i)}=\sqrt{\Gamma_i(\rho_i-1)}I_m$, $l^{(i)}=-I_m$, and $t_p^{(i)}=I_m$.

6. Compute $t_1^{(i)}$ and $t_2^{(i)}$ from \eqref{compute_t} by passing $T^{(i-1)}$, $(\alpha_1,\cdots,\alpha_i)$, and $(\beta_1,\cdots,\beta_i)$.

7. Expand $V_r^{(i)} = \begin{bmatrix}V_r^{(i-1)}&v_i^{\mathrm{fadi}}\end{bmatrix}$, $T^{(i)}=\begin{bmatrix}T^{(i-1)}&t_1^{(i)}\\0&t_2^{(i)}\end{bmatrix}$, $T_p^{(i)} = \mathrm{blkdiag}\big(T_p^{(i-1)},t_p^{(i)}\big)$,\\ $Z_r^{(i)} = \mathrm{blkdiag}\big(Z_r^{(i-1)},z_r^{(i)}\big)$, $\tilde{Z}_r^{(i)}=\begin{bmatrix}\tilde{Z}_r^{(i-1)}&\mathbf{0}\\x_r^{(i)}(t_1^{(i)}\otimes I_m)^* T_p^{(i-1)} Z_r^{(i-1)}&t_2^{(i)}x_r^{(i)}z_r^{(i)}\end{bmatrix}$, \\ $A_r^{(i)}=\begin{bmatrix}A_r^{(i-1)}&\mathbf{0}\\-x_r^{(i)}(l^{(i)})^\top L^{(i-1)}&\beta_iI_m\end{bmatrix}$, $B_r^{(i)}=\begin{bmatrix}B_r^{(i-1)}\\x_r^{(i)}(l^{(i)})^\top\end{bmatrix}$, and $L^{(i)}=\begin{bmatrix}L^{(i-1)}&l^{(i)}\end{bmatrix}$.

8. Update $B_{\perp,i}^{\mathrm{fadi}} \gets B_{\perp,i-1}^{\mathrm{fadi}} -(\alpha_i+\beta_i) E v_i^{\mathrm{fadi}}$, $F^{(i)}\gets\begin{bmatrix}B_{\perp,i}^{\mathrm{fadi}}&EV_r^{(i)}\big(B_r^{(i)}-A_r^{(i)}\tilde{Z}_r^{(i)}(\tilde{Z}_r^{(i)})^\top(L^{(i)})^\top\big)\end{bmatrix}$,
$M^{(i)}\gets\begin{bmatrix}I_m+L^{(i)}\tilde{Z}_r^{(i)}(\tilde{Z}_r^{(i)})^\top(L^{(i)})^\top&I_m\\I_m&\mathbf{0}\end{bmatrix}$, $\Gamma_i\gets \Gamma_i\rho_i$, and $i\gets i+1$.
                
9. \textcolor{blue}{\textbf{If} $\alpha_i\in\mathbb{C}$, $\alpha_{i+1}=\overline{\alpha}_i$, $\beta_i\in\mathbb{C}$, and $\beta_{i+1}=\overline{\beta}_i$ \textbf{do}}

10. Set $x_r^{(i)}$ as in \eqref{x_case2}, $z_r^{(i)}$ as in \eqref{compute_z}, $l^{(i)}=\begin{bmatrix}-I_m&\mathbf{0}\end{bmatrix}$, $t_p^{(i)}=\frac{\beta_i}{\beta_i^2 - 1}
\begin{bmatrix}
\beta_i - \dfrac{\operatorname{Re}(\beta_i)}{|\beta_i|^2} & \dfrac{\operatorname{Im}(\beta_i)}{|\beta_i|^2} \\[8pt]
j\,\dfrac{\operatorname{Im}(\beta_i)}{|\beta_i|^2} & -\dfrac{\operatorname{Im}(\beta_i)}{|\beta_i|^2}
\end{bmatrix}
\otimes I_m$, $t_w^{(i)}=\frac{1}{\beta_i + \overline{\alpha}_i}
\begin{bmatrix}
\operatorname{Re}(\beta_i) + \overline{\alpha}_i & \operatorname{Im}(\beta_i) \\
j\,\operatorname{Im}(\beta_i) & -\operatorname{Im}(\beta_i)
\end{bmatrix}
\otimes I_m$, and $A_{22}=-x_r^{(i)}\begin{bmatrix}-\mathrm{Re}(\beta_i)I_m& \mathrm{Im}(\beta_i)I_m\\ -\mathrm{Im}(\beta_i)I_m& -\mathrm{Re}(\beta_i)I_m\end{bmatrix}^\top(x_r^{(i)})^{-1}$.

11. Compute $t_1^{(i)}$ and $t_2^{(i)}$ from \eqref{compute_t} by passing $T^{(i-1)}$, $(\alpha_1,\cdots,\alpha_{i+1})$, and $(\beta_1,\cdots,\beta_{i+1})$.

12. Expand $V_r^{(i)} = \begin{bmatrix}V_r^{(i-1)}&\mathrm{Re}(v_i^{\mathrm{fadi}})&\mathrm{Im}(v_i^{\mathrm{fadi}})\end{bmatrix}$, $T^{(i)}=\begin{bmatrix}T^{(i-1)}&t_1^{(i)}\\0&t_2^{(i)}\end{bmatrix}$, $T_p^{(i)} = \mathrm{blkdiag}\big(T_p^{(i-1)},t_p^{(i)}\big)$, $Z_r^{(i)} = \mathrm{blkdiag}\big(Z_r^{(i-1)},z_r^{(i)}\big)$, $\tilde{Z}_r^{(i)}=\begin{bmatrix}\tilde{Z}_r^{(i-1)}&\mathbf{0}\\x_r^{(i)}(t_w^{(i)})^*(t_1^{(i)}\otimes I_m)^* T_p^{(i-1)} Z_r^{(i-1)}&x_r^{(i)}(t_w^{(i)})^*(t_2^{(i)}\otimes I_m)^*t_p^{(i)}z_r^{(i)}\end{bmatrix}$,\\ $A_r^{(i)}=\begin{bmatrix}A_r^{(i-1)}&\mathbf{0}\\-x_r^{(i)}(l^{(i)})^\top L^{(i-1)}&A_{22}\end{bmatrix}$, $B_r^{(i)}=\begin{bmatrix}B_r^{(i-1)}\\x_r^{(i)}(l^{(i)})^\top\end{bmatrix}$, and $L^{(i)}=\begin{bmatrix}L^{(i-1)}&l^{(i)}\end{bmatrix}$.

13. Update $B_{\perp,i}^{\mathrm{fadi}} \gets B_{\perp,i-1}^{\mathrm{fadi}} -E \begin{bmatrix}\mathrm{Re}(v_i^{\mathrm{fadi}})&\mathrm{Im}(v_i^{\mathrm{fadi}})\end{bmatrix}x_r^{(i)}(l^{(i)})^\top$,\\ $F^{(i)}\gets\begin{bmatrix}B_{\perp,i}^{\mathrm{fadi}}&EV_r^{(i)}\big(B_r^{(i)}-A_r^{(i)}\tilde{Z}_r^{(i)}(\tilde{Z}_r^{(i)})^\top(L^{(i)})^\top\big)\end{bmatrix}$,
$M^{(i)}\gets\begin{bmatrix}I_m+L^{(i)}\tilde{Z}_r^{(i)}(\tilde{Z}_r^{(i)})^\top(L^{(i)})^\top&I_m\\I_m&\mathbf{0}\end{bmatrix}$,\\ $\Gamma_i\gets \Gamma_i\rho_i^2$, and $i\gets i+2$.

14. \textcolor{blue}{\textbf{If} $\alpha_i\in\mathbb{C}$, $\alpha_{i+1}=\overline{\alpha}_i$, $\beta_i\in\mathbb{R}$, and $\beta_{i+1}\in\mathbb{R}$ \textbf{do}}

15. Set $x_r^{(i)}$ as in \eqref{x_case3}, $z_r^{(i)}=\sqrt{\Gamma_i(\rho_i-1)}I_m$, $\rho_{i+1} = |\beta_{i+1}|^{-2}$, $\Gamma_{i+1}= \Gamma_i\rho_{i}$, $z_r^{(i+1)}=\sqrt{\Gamma_{i+1}(\rho_{i+1}-1)}I_m$, $\hat{z}_r=\mathrm{blkdiag}(z_r^{(i)},z_r^{(i+1)})$, $l^{(i)}=\begin{bmatrix}-I_m&\mathbf{0}\end{bmatrix}$, $t_p^{(i)}=\begin{bmatrix}1 & 0\\
0& 1\end{bmatrix}\otimes I_m$, $t_w^{(i)}= \begin{bmatrix}1& \dfrac{1}{(\beta_{i+1} + \overline{\alpha}_i)}\\
            0& -\dfrac{1}{(\beta_{i+1} + \overline{\alpha}_i)}\end{bmatrix}\otimes I_m$, and $A_{22}=-x_r^{(i)}\begin{bmatrix}-\beta_iI_m& I_m\\ \mathbf{0}& -\beta_{i+1}I_m\end{bmatrix}^\top(x_r^{(i)})^{-1}$.
            
16. Compute $t_1^{(i)}$ and $t_2^{(i)}$ from \eqref{compute_t} by passing $T^{(i-1)}$, $(\alpha_1,\cdots,\alpha_{i+1})$, and $(\beta_1,\cdots,\beta_{i+1})$.

17. Expand $V_r^{(i)} = \begin{bmatrix}V_r^{(i-1)}&\mathrm{Re}(v_i^{\mathrm{fadi}})&\mathrm{Im}(v_i^{\mathrm{fadi}})\end{bmatrix}$, $T^{(i)}=\begin{bmatrix}T^{(i-1)}&t_1^{(i)}\\0&t_2^{(i)}\end{bmatrix}$, $T_p^{(i)} = \mathrm{blkdiag}\big(T_p^{(i-1)},t_p^{(i)}\big)$, $Z_r^{(i)} = \mathrm{blkdiag}\big(Z_r^{(i-1)},\hat{z}_r\big)$, $\tilde{Z}_r^{(i)}=\begin{bmatrix}\tilde{Z}_r^{(i-1)}&\mathbf{0}\\x_r^{(i)}(t_w^{(i)})^*(t_1^{(i)}\otimes I_m)^* T_p^{(i-1)} Z_r^{(i-1)}&x_r^{(i)}(t_w^{(i)})^*(t_2^{(i)}\otimes I_m)^*t_p^{(i)}\hat{z}_r\end{bmatrix}$,\\ $A_r^{(i)}=\begin{bmatrix}A_r^{(i-1)}&\mathbf{0}\\-x_r^{(i)}(l^{(i)})^\top L^{(i-1)}&A_{22}\end{bmatrix}$, $B_r^{(i)}=\begin{bmatrix}B_r^{(i-1)}\\x_r^{(i)}(l^{(i)})^\top\end{bmatrix}$, and $L^{(i)}=\begin{bmatrix}L^{(i-1)}&l^{(i)}\end{bmatrix}$.

18. Update $B_{\perp,i}^{\mathrm{fadi}} \gets B_{\perp,i-1}^{\mathrm{fadi}} -E \begin{bmatrix}\mathrm{Re}(v_i^{\mathrm{fadi}})&\mathrm{Im}(v_i^{\mathrm{fadi}})\end{bmatrix}x_r^{(i)}(l^{(i)})^\top$,\\ $F^{(i)}\gets\begin{bmatrix}B_{\perp,i}^{\mathrm{fadi}}&EV_r^{(i)}\big(B_r^{(i)}-A_r^{(i)}\tilde{Z}_r^{(i)}(\tilde{Z}_r^{(i)})^\top(L^{(i)})^\top\big)\end{bmatrix}$,
$M^{(i)}\gets\begin{bmatrix}I_m+L^{(i)}\tilde{Z}_r^{(i)}(\tilde{Z}_r^{(i)})^\top(L^{(i)})^\top&I_m\\I_m&\mathbf{0}\end{bmatrix}$,\\ $\Gamma_i\gets \Gamma_{i+1}\rho_{i+1}$, and $i\gets i+2$.

19. \textcolor{blue}{\textbf{If} $\alpha_i\in\mathbb{R}$, $\alpha_{i+1}\in\mathbb{R}$, $\beta_i\in\mathbb{C}$, and $\beta_{i+1}=\overline{\beta}_i$ \textbf{do}}

20. Solve $\big(A + \alpha_{i+1} E)v_{r}^{i+1}= Ev_i^{\mathrm{fadi}}$ for $v_{r}^{i+1}$.

21. Set $x_r^{(i)}$ as in \eqref{x_case4}, $z_r^{(i)}$ as in \eqref{compute_z}, $l^{(i)}=\begin{bmatrix}-I_m&\mathbf{0}\end{bmatrix}$, $t_p^{(i)}=\frac{\beta_i}{\beta_i^2 - 1}
\begin{bmatrix}
\beta_i - \dfrac{\operatorname{Re}(\beta_i)}{|\beta_i|^2} & \dfrac{\operatorname{Im}(\beta_i)}{|\beta_i|^2} \\[8pt]
j\,\dfrac{\operatorname{Im}(\beta_i)}{|\beta_i|^2} & -\dfrac{\operatorname{Im}(\beta_i)}{|\beta_i|^2}
\end{bmatrix}
\otimes I_m$, $t_w^{(i)}=\frac{1}{\beta_i + \overline{\alpha}_i}
\begin{bmatrix}
\operatorname{Re}(\beta_i) + \overline{\alpha}_i & \operatorname{Im}(\beta_i) \\
j\,\operatorname{Im}(\beta_i) & -\operatorname{Im}(\beta_i)
\end{bmatrix}
\otimes I_m$, and $A_{22}=-x_r^{(i)}\begin{bmatrix}-\mathrm{Re}(\beta_i)I_m& \mathrm{Im}(\beta_i)I_m\\ -\mathrm{Im}(\beta_i)I_m& -\mathrm{Re}(\beta_i)I_m\end{bmatrix}^\top(x_r^{(i)})^{-1}$.

22. Compute $t_1^{(i)}$ and $t_2^{(i)}$ from \eqref{compute_t} by passing $T^{(i-1)}$, $(\alpha_1,\cdots,\alpha_{i+1})$, and $(\beta_1,\cdots,\beta_{i+1})$.

23. Expand $V_r^{(i)} = \begin{bmatrix}V_r^{(i-1)}&v_i^{\mathrm{fadi}}&v_{r}^{i+1}\end{bmatrix}$, $T^{(i)}=\begin{bmatrix}T^{(i-1)}&t_1^{(i)}\\0&t_2^{(i)}\end{bmatrix}$, $T_p^{(i)} = \mathrm{blkdiag}\big(T_p^{(i-1)},t_p^{(i)}\big)$,\\ $Z_r^{(i)} = \mathrm{blkdiag}\big(Z_r^{(i-1)},z_r^{(i)}\big)$, $\tilde{Z}_r^{(i)}=\begin{bmatrix}\tilde{Z}_r^{(i-1)}&\mathbf{0}\\x_r^{(i)}(t_w^{(i)})^*(t_1^{(i)}\otimes I_m)^* T_p^{(i-1)} Z_r^{(i-1)}&x_r^{(i)}(t_w^{(i)})^*(t_2^{(i)}\otimes I_m)^*t_p^{(i)}z_r^{(i)}\end{bmatrix}$,\\$A_r^{(i)}=\begin{bmatrix}A_r^{(i-1)}&\mathbf{0}\\-x_r^{(i)}(l^{(i)})^\top L^{(i-1)}&A_{22}\end{bmatrix}$, $B_r^{(i)}=\begin{bmatrix}B_r^{(i-1)}\\x_r^{(i)}(l^{(i)})^\top\end{bmatrix}$, and $L^{(i)}=\begin{bmatrix}L^{(i-1)}&l^{(i)}\end{bmatrix}$.

24. Update $B_{\perp,i}^{\mathrm{fadi}} \gets B_{\perp,i-1}^{\mathrm{fadi}} -E \begin{bmatrix}v_i^{\mathrm{fadi}}&v_{r}^{i+1}\end{bmatrix}x_r^{(i)}(l^{(i)})^\top$,\\ $F^{(i)}\gets\begin{bmatrix}B_{\perp,i}^{\mathrm{fadi}}&EV_r^{(i)}\big(B_r^{(i)}-A_r^{(i)}\tilde{Z}_r^{(i)}(\tilde{Z}_r^{(i)})^\top(L^{(i)})^\top\big)\end{bmatrix}$,
$M^{(i)}\gets\begin{bmatrix}I_m+L^{(i)}\tilde{Z}_r^{(i)}(\tilde{Z}_r^{(i)})^\top(L^{(i)})^\top&I_m\\I_m&\mathbf{0}\end{bmatrix}$,\\ $\Gamma_i\gets \Gamma_i\rho_i^2$, and $i\gets i+2$.

\noindent\rule{\textwidth}{0.8pt}
\subsection{Applications}

This subsection presents two applications of the G-LRCF-ADI method for which the ADI shifts $\alpha_i$ must lie on the unit circle, i.e., $|\alpha_i|=1$. Since the standard LRCF-ADI method requires $|\alpha_i|>1$, it cannot be used for these applications.

\subsubsection{Approximation of Frequency-Limited Gramians}

The frequency-limited controllability Gramian $P_\Omega$ over the frequency interval $\Omega=[-\omega_f,\omega_f]$ rad/sec, with $\omega_f<\pi$, satisfies the Stein equation
\begin{align}
AP_\Omega A^\top-E P_\Omega E^\top
+EF_\Omega(E,A)E^{-1}BB^\top
+BB^\top E^{-\top}F_\Omega(E,A)^\top E^\top
=0,
\end{align}
where
\begin{align}
F_\Omega(E,A)
&=
\frac{1}{4\pi}
\int_{-\omega_f}^{\omega_f}
\big(I-e^{-j\nu}E^{-1}A\big)^{-1}
\big(I+e^{-j\nu}E^{-1}A\big)
\,d\nu
\nonumber\\
&=
\frac{1}{2\pi}
\mathrm{Re}\Big(
\omega_f I
-2j\mathrm{ln}\big(I-e^{-j\omega_f}E^{-1}A\big)
\Big);
\end{align}
see \cite{gawronski1990model,petersson2013nonlinear}.

The Gramian $P_\Omega$ also admits the integral representation
\begin{align}
P_\Omega
=
\frac{1}{2\pi}
\int_{-\omega_f}^{\omega_f}
\big(e^{j\nu}E-A\big)^{-1}
BB^\top
\big(e^{j\nu}E-A\big)^{-*}
\,d\nu.
\label{fl_int}
\end{align}
Moreover, $P_\Omega$ and the standard controllability Gramian $P$ are related by
\[
P_\Omega
=
F_\Omega(E,A)P
+
PF_\Omega(E,A)^\top.
\]

By choosing the ADI shifts as
\[
\alpha_i=-e^{j\omega_i},
\qquad
\omega_i\in\Omega,
\]
the G-LRCF-ADI method can be used to perform numerical integration of \eqref{fl_int}. In numerical integration, the integrand is replaced by an interpolant at selected nodes and the interpolant is integrated instead of the original integrand. Here, the shifts define interpolation nodes on the unit circle. Since the G-LRCF-ADI method allows shifts on the unit circle, unlike the standard LRCF-ADI method, it can approximate the integrand in \eqref{fl_int} by the reduced-order interpolant. The poles $\beta_i$ can still be selected using the procedure in Subsection \ref{sub1}.

Consequently, the frequency-limited Gramian can be approximated as
\begin{align}
P_\Omega\approx \hat{P}_\Omega^{(k)}
&=V_r^{(k)}\Bigg(\frac{1}{2\pi}\int_{-\omega_f}^{\omega_f}
\big(e^{j\nu}I-A_r^{(k)}\big)^{-1}
B_r^{(k)}(B_r^{(k)})^\top
\big(e^{j\nu}I-A_r^{(k)}\big)^{-*}d\nu\Bigg)
(V_r^{(k)})^\top\nonumber\\
&=V_r^{(k)}\Big(
F_\Omega(I,A_r^{(k)})Z_r^{(k)}(Z_r^{(k)})^\top
+Z_r^{(k)}(Z_r^{(k)})^\top F_\Omega(I,A_r^{(k)})^\top
\Big)(V_r^{(k)})^\top\nonumber\\
&=V_r^{(k)}\tilde{P}_\Omega^{(k)}(V_r^{(k)})^\top,
\end{align}
where $\tilde{P}_\Omega^{(k)}$ is the frequency-limited controllability Gramian of the reduced pair $(A_r^{(k)},B_r^{(k)})$. Since $(A_r^{(k)},B_r^{(k)})$ is stable and controllable for $|\beta_i|<1$ and $\beta_i\neq0$, $\tilde{P}_\Omega^{(k)}$ is positive definite and admits a Cholesky factorization. This is useful for frequency-limited balanced truncation (FLBT).

An additional advantage of the G-LRCF-ADI method is that it avoids computing $F_\Omega(E,A)$, which is not feasible in large-scale settings. Its interpolatory properties instead provide the approximation
\begin{align}
F_\Omega(E,A)E^{-1}B
\approx
G_\Omega^{(k)}
=
V_r^{(k)}
F_\Omega(I,A_r^{(k)})
B_r^{(k)}.
\end{align}

Since $A_r^{(k)}$ has the lower block-triangular form
\[
A_r^{(k)}
=
\begin{bmatrix}
a^{(k-1)} & \mathbf{0}\\
a_{21}^{(k)} & a_{22}^{(k)}
\end{bmatrix}
\otimes I_m,
\]
$F_\Omega(I,A_r^{(k)})$ also has a lower block-triangular form:
\[
F_\Omega(I,A_r^{(k)})
=
\begin{bmatrix}
F_\Omega(I,a^{(k-1)}) & \mathbf{0}\\
f_{21}^{(k)} & F_\Omega(I,a_{22}^{(k)})
\end{bmatrix}
\otimes I_m;
\]
see \cite{petersson2014model}. Furthermore, the commutativity relation
\[
A_r^{(k)}F_\Omega(I,A_r^{(k)})
=
F_\Omega(I,A_r^{(k)})A_r^{(k)}
\]
implies that $f_{21}^{(k)}$ satisfies
\begin{align}
a_{22}^{(k)}f_{21}^{(k)}
-f_{21}^{(k)}a^{(k-1)}
+a_{21}^{(k)}F_\Omega(I,a^{(k-1)})
-F_\Omega(I,a_{22}^{(k)})a_{21}^{(k)}
=0;
\end{align}
see \cite{zulfiqar2022adaptive}. This Sylvester equation is inexpensive to solve because $a^{(k-1)}$ and $a_{22}^{(k)}$ are small. Since $F_\Omega(I,a_{22}^{(k)})$ can also be evaluated cheaply, $F_\Omega(I,A_r^{(k)})$ can be computed recursively.

Once the relative change in $G_\Omega^{(k)}$ stagnates, namely,
\[
\frac{\|G_\Omega^{(k)}-G_\Omega^{(k-1)}\|}
{\|G_\Omega^{(k-1)}\|}
<\tau,
\]
the residual
\begin{align}
R_\Omega^{(k)}
&=
A\hat{P}_\Omega^{(k)}A^\top
-E\hat{P}_\Omega^{(k)}E^\top
+EG_\Omega^{(k)}B^\top
+B(G_\Omega^{(k)})^\top E^\top
\nonumber\\
&=
\begin{bmatrix}
AV_r^{(k)} & EV_r^{(k)} & B & EG_\Omega^{(k)}
\end{bmatrix}
\begin{bmatrix}
\tilde{P}_\Omega^{(k)} & 0 & 0 & 0\\
0 & -\tilde{P}_\Omega^{(k)} & 0 & 0\\
0 & 0 & 0 & I\\
0 & 0 & I & 0
\end{bmatrix}
\begin{bmatrix}
AV_r^{(k)} & EV_r^{(k)} & B & EG_\Omega^{(k)}
\end{bmatrix}^\top
\nonumber\\
&=
F_\Omega^{(k)}M_\Omega^{(k)}(F_\Omega^{(k)})^\top
\end{align}
can be used to assess the accuracy of $\hat{P}_\Omega^{(k)}$.
\subsubsection{Data-Driven Balanced Truncation}

It is shown in \cite{zulfiqar2026new} that, for $\sigma_i=-\alpha_i$ and distinct shifts $\{\alpha_i\}_{i=1}^{k}$, $V_{\mathrm{kry}}$ and $V_{\mathrm{fadi}}^{(k)}$ are related by
\[
V_{\mathrm{fadi}}^{(k)}
=
V_{\mathrm{kry}}T_{v,\mathrm{kry}}^{(k)},
\]
where
\begin{equation}
\label{Tv_kry}
T_{v,\mathrm{kry}}^{(k)}(j,i)
=
\begin{cases}
(-1)^{i}
\dfrac{
\prod_{q=1}^{i-1}(\alpha_j+\beta_q)
}{
\prod_{\substack{q=1\\q\neq j}}^{i}(\alpha_q-\alpha_j)
},
& 1\leq j\leq i,
\\[10pt]
0,
& j>i,
\end{cases}
\otimes I_m.
\end{equation}

Let $\{\hat{\alpha}_i\}_{i=1}^k$ and $\{\hat{\beta}_i\}_{i=1}^k$ denote the ADI shifts used to approximate
\[
Q\approx
\big(W_r^{(k)}Y_r^{(k)}\big)
\big(W_r^{(k)}Y_r^{(k)}\big)^\top
\]
using G-LRCF-ADI, where the shifts $\hat{\alpha}_i$ are distinct. Define
\[
W_{\mathrm{kry}}
=
\begin{bmatrix}
(-\overline{\hat{\alpha}}_1E^\top-A^\top)^{-1}C^\top
&
\cdots
&
(-\overline{\hat{\alpha}}_kE^\top-A^\top)^{-1}C^\top
\end{bmatrix}.
\]
Next, define $T_{w,\mathrm{kry}}^{(k)}$ by substituting
\[
\alpha_i=\overline{\hat{\alpha}}_i,
\qquad
\beta_i=\overline{\hat{\beta}}_i
\]
in \eqref{Tv_kry}. Similarly, define $\hat{T}_w^{(k)}$ by making the same substitutions in the expression for $T_v^{(k)}$.

The G-LRCF-ADI-based low-rank BT method replaces the exact Cholesky factors $Z_p$ and $Z_q$ of $P$ and $Q$ in the BSA by the approximations
\[
Z_p
\approx
V_{\mathrm{kry}}
T_{v,\mathrm{kry}}^{(k)}
T_v^{(k)}
\tilde{Z}_r^{(k)}
\]
and
\[
Z_q
\approx
W_{\mathrm{kry}}
T_{w,\mathrm{kry}}^{(k)}
\hat{T}_w^{(k)}
\tilde{Y}_r^{(k)},
\]
respectively.

Compute the singular value decomposition
\[
(\tilde{Y}_r^{(k)})^\top
(\hat{T}_w^{(k)})^*
(T_{w,\mathrm{kry}}^{(k)})^*
(W_{\mathrm{kry}})^*
EV_{\mathrm{kry}}
T_{v,\mathrm{kry}}^{(k)}
T_v^{(k)}
\tilde{Z}_r^{(k)}
=
\begin{bmatrix}
\tilde{U}_q & \tilde{U}_2
\end{bmatrix}
\begin{bmatrix}
\tilde{\Sigma}_q & \mathbf{0}\\
\mathbf{0} & \tilde{\Sigma}_2
\end{bmatrix}
\begin{bmatrix}
\tilde{V}_q^\top\\
\tilde{V}_2^\top
\end{bmatrix}.
\]
The projection matrices with $q$ columns are then given by
\begin{align}
W_q
&=
W_{\mathrm{kry}}
T_{w,\mathrm{kry}}^{(k)}
\hat{T}_w^{(k)}
\tilde{Y}_r^{(k)}
\tilde{U}_q
\tilde{\Sigma}_q^{-1/2},
\nonumber\\
V_q
&=
V_{\mathrm{kry}}
T_{v,\mathrm{kry}}^{(k)}
T_v^{(k)}
\tilde{Z}_r^{(k)}
\tilde{V}_q
\tilde{\Sigma}_q^{-1/2}.
\label{bsa_proj}
\end{align}

The $q$th-order ROM obtained by low-rank BT is
\begin{align}
\tilde{A}_q
&=
\tilde{\Sigma}_q^{-1/2}
(\tilde{U}_q)^\top
(\tilde{Y}_r^{(k)})^\top
(\hat{T}_w^{(k)})^*
(T_{w,\mathrm{kry}}^{(k)})^*
(W_{\mathrm{kry}})^*
AV_{\mathrm{kry}}
T_{v,\mathrm{kry}}^{(k)}
T_v^{(k)}
\tilde{Z}_r^{(k)}
\tilde{V}_q
\tilde{\Sigma}_q^{-1/2},
\nonumber\\
\tilde{B}_q
&=
\tilde{\Sigma}_q^{-1/2}
(\tilde{U}_q)^\top
(\tilde{Y}_r^{(k)})^\top
(\hat{T}_w^{(k)})^*
(T_{w,\mathrm{kry}}^{(k)})^*
(W_{\mathrm{kry}})^*
B,
\nonumber\\
\tilde{C}_q
&=
CV_{\mathrm{kry}}
T_{v,\mathrm{kry}}^{(k)}
T_v^{(k)}
\tilde{Z}_r^{(k)}
\tilde{V}_q
\tilde{\Sigma}_q^{-1/2}.
\end{align}

When $\alpha_i$ and $\hat{\alpha}_i$ lie on the unit circle, the terms
$(W_{\mathrm{kry}})^*EV_{\mathrm{kry}}$,
$(W_{\mathrm{kry}})^*AV_{\mathrm{kry}}$,
$(W_{\mathrm{kry}})^*B$, and
$CV_{\mathrm{kry}}$
can be obtained experimentally from frequency-response samples
$G(e^{j\omega_i})$. In particular,
\begin{align}
\tilde{E}(j,i)
&=
(W_{\mathrm{kry}}(:,i))^*
EV_{\mathrm{kry}}(:,j)
=
\frac{
G(-\hat{\alpha}_i)-G(-\alpha_j)
}{
\hat{\alpha}_i-\alpha_j
},
\nonumber\\
\tilde{A}(j,i)
&=
(W_{\mathrm{kry}}(:,i))^*
AV_{\mathrm{kry}}(:,j)
=
\frac{
\alpha_jG(-\alpha_j)
-\hat{\alpha}_iG(-\hat{\alpha}_i)
}{
\hat{\alpha}_i-\alpha_j
},
\nonumber\\
\tilde{B}(j,:)
&=
(W_{\mathrm{kry}}(:,j))^*B
=
G(-\hat{\alpha}_j),
\nonumber\\
\tilde{C}(:,i)
&=
CV_{\mathrm{kry}}(:,i)
=
G(-\alpha_i).
\end{align}
All remaining terms are computed analytically from the ADI shifts
$(\alpha_i,\beta_i)$ and $(\hat{\alpha}_i,\hat{\beta}_i)$. Therefore, the ROM
\[
\tilde{G}_q(z)
=
\tilde{C}_q(zI-\tilde{A}_q)^{-1}\tilde{B}_q
\]
can be computed non-intrusively from samples of $G(z)$ on the unit circle. FLBT can be implemented similarly using frequency-response data and is omitted here for brevity.
\section{Numerical Results}

This section evaluates the numerical performance of the G-LRCF-ADI method using large-scale state-space models. The data-driven implementations of BT and FLBT based on G-LRCF-ADI are also tested. MATLAB codes for reproducing the results in this section are publicly available at \cite{mycodes}. All experiments are performed in MATLAB R2025b on a Windows 11 laptop with 32 GB RAM and an Intel(R) Core(TM) Ultra 9 285H processor running at 2.9 GHz.

\subsection{Model Description\label{sub2}}

The large-scale state-space models used in this section are constructed as follows. First, generate $n_r$ distinct real poles $\eta_i$ and $2n_c$ distinct complex poles $\zeta_i$, forming $n_c$ complex-conjugate pole pairs. Next, construct arbitrary matrices
\[
B_1=
\begin{bmatrix}
b_1\\
\vdots\\
b_{n_c}
\end{bmatrix}
\in\mathbb{R}^{2n_c\times m},
\qquad
C_1=
\begin{bmatrix}
c_1 & \cdots & c_{n_c}
\end{bmatrix}
\in\mathbb{R}^{p\times 2n_c}.
\]

Assume that the first $2n_d$ poles, corresponding to $n_d$ complex-conjugate pairs, are dominant. The matrices $B_1$ and $C_1$ are scaled as
\[
B_1=
\begin{bmatrix}
\sqrt{10\frac{1-|\zeta_1|^2}{\|c_1b_1\|_2}}b_1\\
\vdots\\
\sqrt{10\frac{1-|\zeta_{n_d}|^2}{\|c_{n_d}b_{n_d}\|_2}}b_{n_d}\\
\sqrt{0.1\frac{1-|\zeta_{n_d+1}|^2}{\|c_{n_d+1}b_{n_d+1}\|_2}}b_{n_d+1}\\
\vdots\\
\sqrt{0.1\frac{1-|\zeta_{n_c}|^2}{\|c_{n_c}b_{n_c}\|_2}}b_{n_c}
\end{bmatrix},\quad
\text{and}\quad
C_1=
\begin{bmatrix}
\sqrt{10\frac{1-|\zeta_1|^2}{\|c_1b_1\|_2}}c_1^\top\\
\vdots\\
\sqrt{10\frac{1-|\zeta_{n_d}|^2}{\|c_{n_d}b_{n_d}\|_2}}c_{n_d}^\top\\
\sqrt{0.1\frac{1-|\zeta_{n_d+1}|^2}{\|c_{n_d+1}b_{n_d+1}\|_2}}c_{n_d+1}^\top\\
\vdots\\
\sqrt{0.1\frac{1-|\zeta_{n_c}|^2}{\|c_{n_c}b_{n_c}\|_2}}c_{n_c}^\top
\end{bmatrix}^\top.
\]

Next, define
\[
B_2=
0.001
\left(
\min_{i=1,\ldots,n_r}|\eta_i|
\right)^2
\mathbf{1}_{n_r\times m},
\qquad
C_2=
0.001
\left(
\min_{i=1,\ldots,n_r}|\eta_i|
\right)^2
\mathbf{1}_{p\times n_r}.
\]
The matrices associated with the complex and real poles are defined by
\[
A_1=
\mathrm{blkdiag}\left(
\begin{bmatrix}
\mathrm{Re}(\zeta_1) & \mathrm{Im}(\zeta_1)\\
-\mathrm{Im}(\zeta_1) & \mathrm{Re}(\zeta_1)
\end{bmatrix},
\ldots,
\begin{bmatrix}
\mathrm{Re}(\zeta_{n_c}) & \mathrm{Im}(\zeta_{n_c})\\
-\mathrm{Im}(\zeta_{n_c}) & \mathrm{Re}(\zeta_{n_c})
\end{bmatrix}
\right),\quad
\text{and}\quad
A_2=\mathrm{diag}(\eta_1,\ldots,\eta_{n_r}).
\]

Define
\[
\mathcal{P}
=
\operatorname{tridiag}
\left(
\frac{1}{2},\,1,\,\frac{1}{3}
\right)
\in\mathbb{R}^{(n_r+2n_c)\times(n_r+2n_c)}\quad
\text{and}\quad
\mathcal{Q}
=
\operatorname{tridiag}
\left(
\frac{1}{4},\,1,\,\frac{1}{6}
\right)
\in\mathbb{R}^{(n_r+2n_c)\times(n_r+2n_c)}.
\]
A state-space realization of order $n=n_r+2n_c$, with $m$ inputs and $p$ outputs, is then constructed as
\[
E=\mathcal{Q}\mathcal{P},
\qquad
A=\mathcal{Q}\mathrm{blkdiag}(A_1,A_2)\mathcal{P},
\qquad
B=\mathcal{Q}
\begin{bmatrix}
B_1\\
B_2
\end{bmatrix},
\qquad
C=
\begin{bmatrix}
C_1 & C_2
\end{bmatrix}
\mathcal{P}.
\]

The matrices $E$ and $A$ are sparse by construction. The real poles are chosen to be nearly uncontrollable and nearly unobservable, whereas the first $2n_d$ complex poles are the most controllable and observable poles. By varying the number of dominant poles, the numerical ranks of $P$ and $Q$, as well as the number of significant Hankel singular values, can be controlled.

The model order, pole locations, and numbers of inputs and outputs can all be selected by the user. Thus, this construction provides large-scale benchmark models with prescribed properties for testing MOR algorithms and low-rank matrix-equation solvers. An additional advantage is that the poles are known by construction, even when computing the eigenvalue decomposition of $E^{-1}A$ is infeasible.
\subsection{Example 1: Automatic Shift Generation}

This example evaluates the automatic shift-generation method proposed in Subsection \ref{sub1}. The matrices in the Stein equation \eqref{lyap_p} are generated using the procedure described in Subsection \ref{sub2}. The matrix dimensions are
\[
E\in\mathbb{R}^{10^7\times 10^7},
\qquad
A\in\mathbb{R}^{10^7\times 10^7},
\qquad
B\in\mathbb{R}^{10^7\times 2}.
\]
The parameters are set to $n_c=25$ and $n_d=10$. Hence, the numerical rank of $P$ is expected to be approximately $20$, making the LRCF-ADI method suitable for this problem. All poles lie inside the unit circle, ensuring that the Stein equation \eqref{lyap_p} has a unique solution. The pole locations are shown in Figure \ref{fig1}.

\begin{figure}[!h]
  \centering
  \includegraphics[width=10cm]{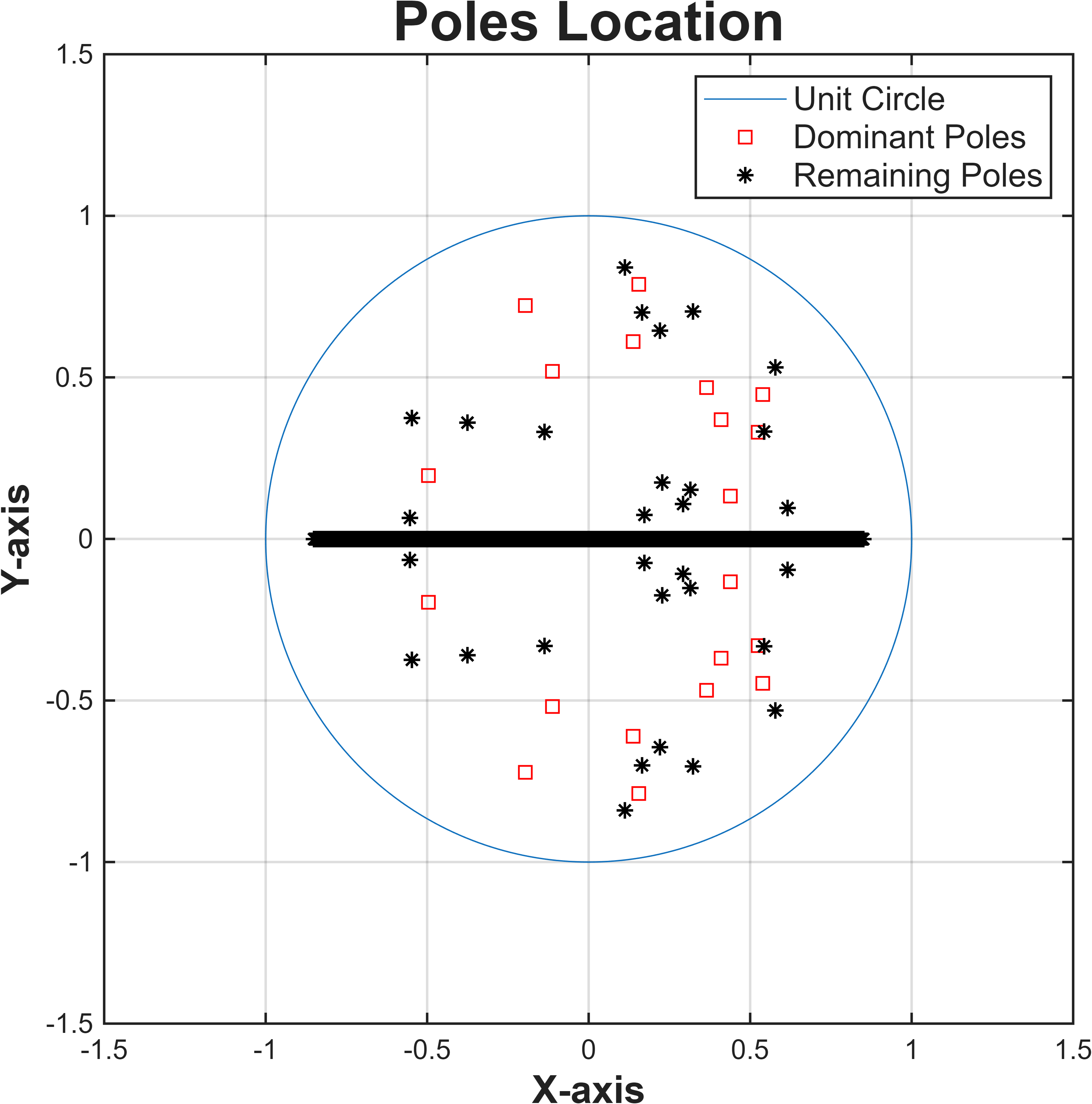}
  \caption{Eigenvalues of $E^{-1}A$}
  \label{fig1}
\end{figure}

The LRCF-ADI method is considered converged when
\[
\frac{
\|AV_{\mathrm{fadi}}^{(k)}Z_{\mathrm{fadi}}^{(k)}(Z_{\mathrm{fadi}}^{(k)})^\top (V_{\mathrm{fadi}}^{(k)})^\top A^\top
-EV_{\mathrm{fadi}}^{(k)}Z_{\mathrm{fadi}}^{(k)}(Z_{\mathrm{fadi}}^{(k)})^\top (V_{\mathrm{fadi}}^{(k)})^\top E^\top
+BB^\top\|_2
}{
\|BB^\top\|_2
}
<10^{-10}.
\]
First, the shifts are selected as
\[
\alpha_i=-\frac{1}{\overline{\zeta}_i},
\]
and the resulting performance is compared with that of the proposed automatic shift-generation method, which aims to estimate these poles. For the automatic method, the initial shift is set to $\alpha_1=2$. The decay of the normalized spectral-norm residual is shown in Figure \ref{fig2}.

\begin{figure}[!h]
  \centering
  \includegraphics[width=10cm]{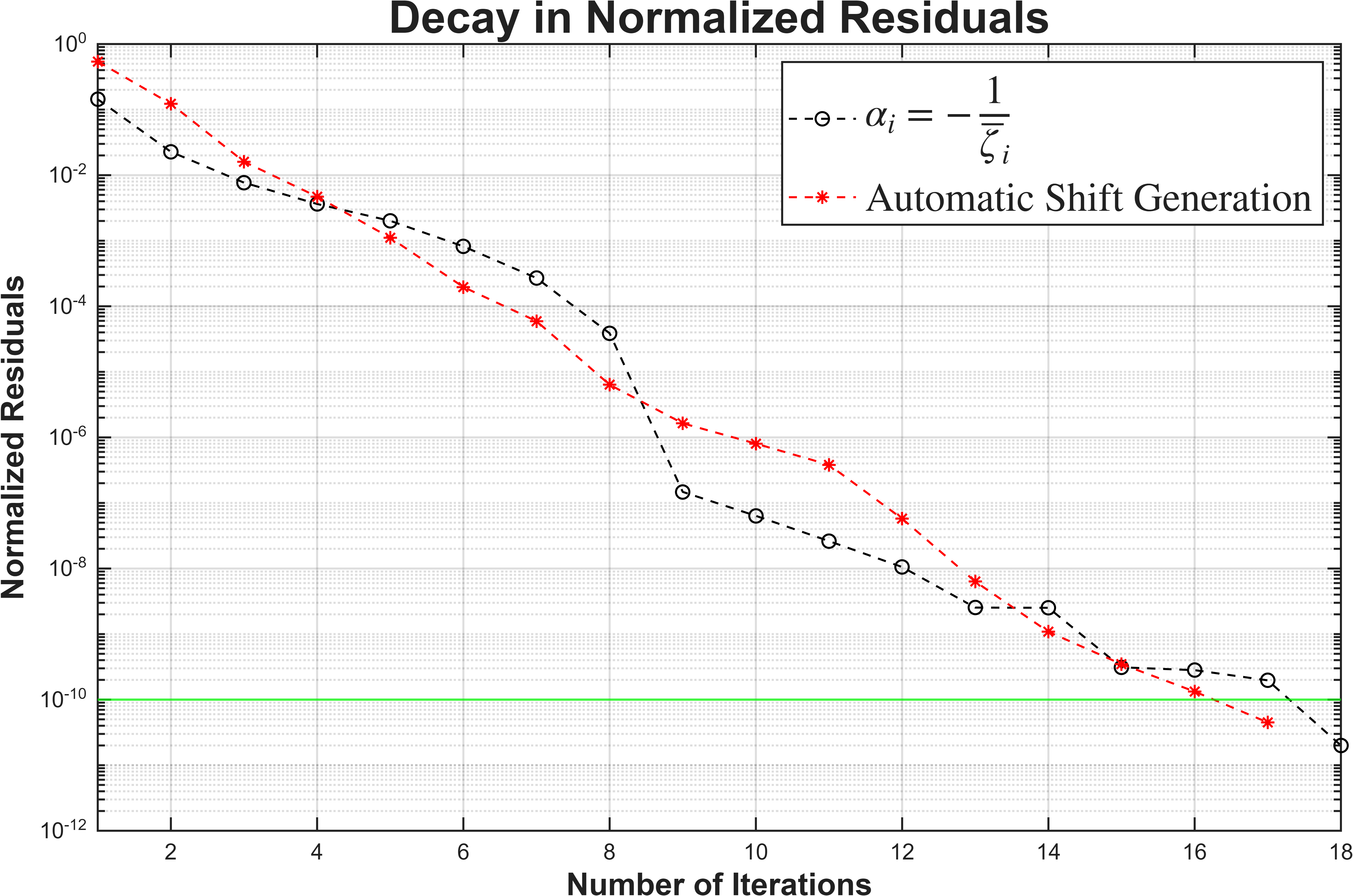}
  \caption{Decay of the normalized residual}
  \label{fig2}
\end{figure}

The proposed automatic shift-generation method provides an accurate approximation of $P$. The LRCF-ADI method requires $75.2560$ seconds when using automatic shifts, compared with $50.5407$ seconds when using the shifts $\alpha_i=-1/\overline{\zeta}_i$. Figure \ref{fig3} compares the eigenvalues of $\tilde{A}^{(k)}$ in the projected Stein equation \eqref{proj_lyap}.

\begin{figure}[!h]
  \centering
  \includegraphics[width=10cm]{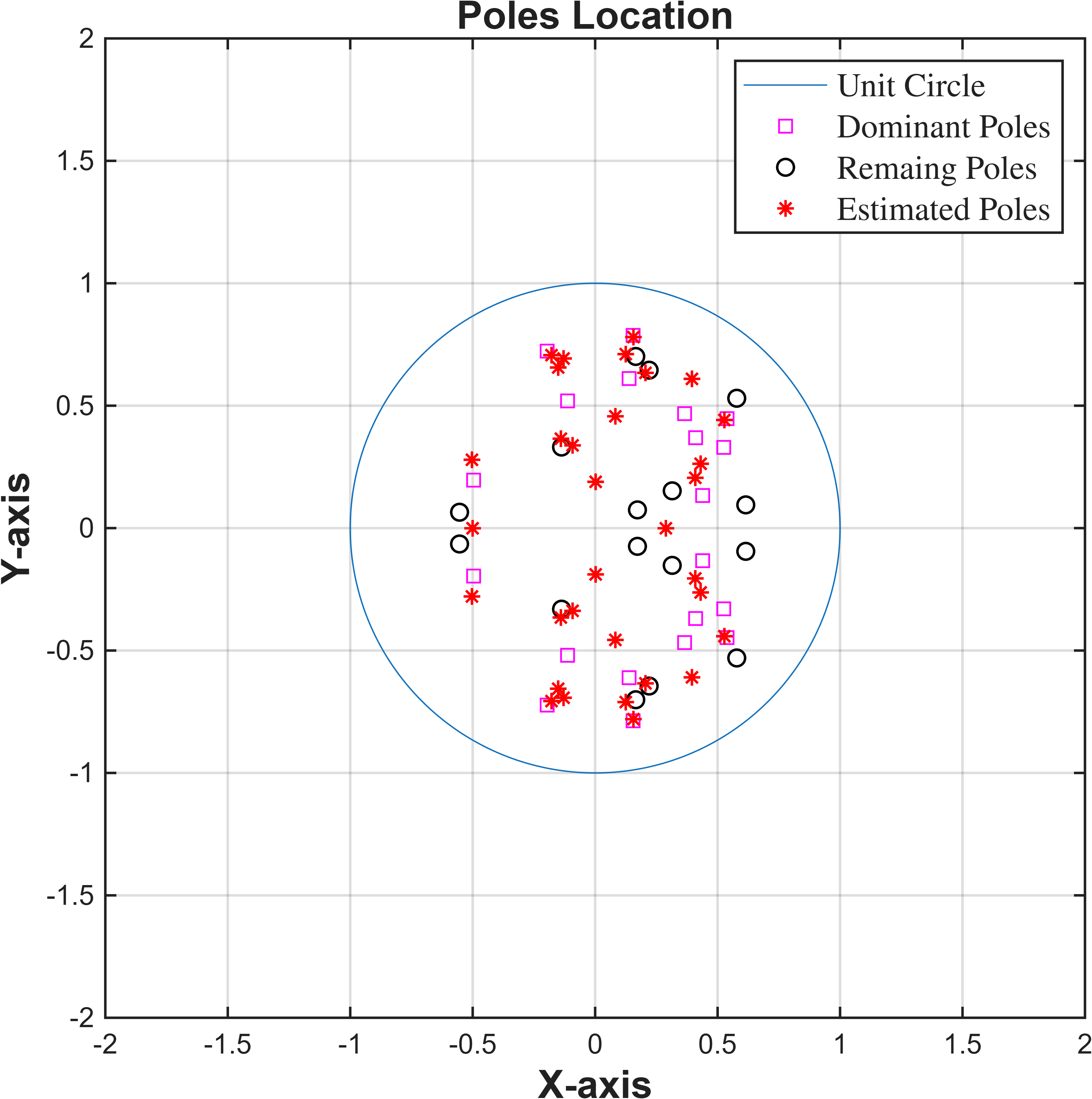}
  \caption{Pole locations of $\tilde{A}^{(k)}$}
  \label{fig3}
\end{figure}

Figure \ref{fig3} shows that the proposed automatic shift-generation method estimates several poles $\zeta_i$ closely.
\subsection{Example 2: $\mathcal{H}_2$ Pseudo-Optimality}

Consider the model from the previous example. In this example, the shifts in G-LRCF-ADI are selected as
\[
\{\alpha_i\}_{i=1}^{k}
=
\left\{
-\frac{1}{\overline{\zeta}_i}
\right\}_{i=1}^{k},
\qquad
\{\beta_i\}_{i=1}^{k}
=
\{\zeta_{k-i+1}\}_{i=1}^{k}.
\]
Thus, G-LRCF-ADI does not perform $\mathcal{H}_2$-pseudo-optimal MOR at every iteration. However, at the final iteration, the shifts enforce $\mathcal{H}_2$ pseudo-optimality. Therefore, the final approximation must coincide with that obtained by the LRCF-ADI method.

Figure \ref{fig4} shows the decay of the normalized spectral-norm residual. Although G-LRCF-ADI behaves differently from LRCF-ADI during the intermediate iterations, the two methods coincide at the final iteration, as predicted by the theory. G-LRCF-ADI converges in $87.5413$ seconds.

\begin{figure}[!h]
  \centering
  \includegraphics[width=10cm]{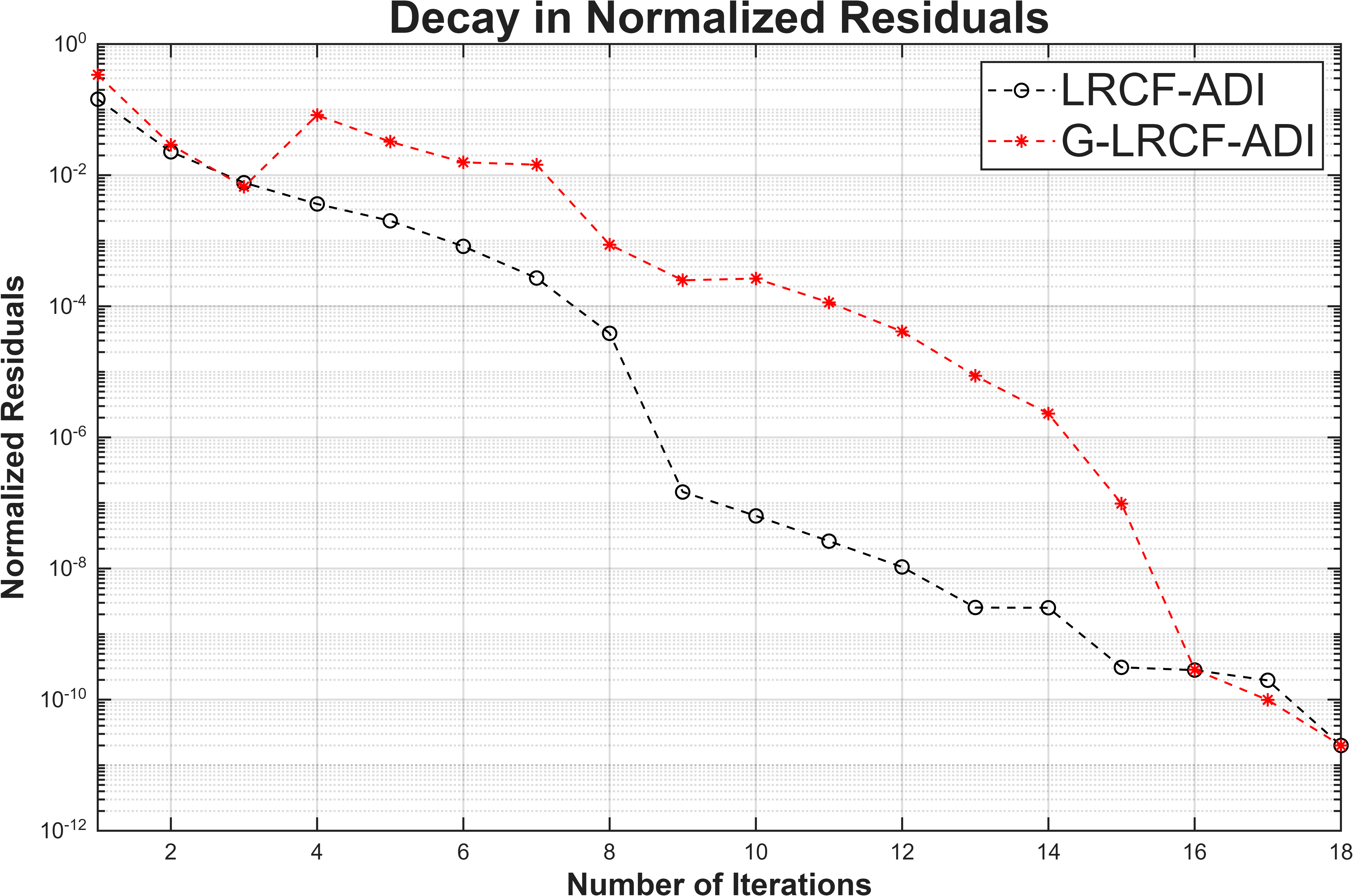}
  \caption{Decay of the normalized residual}
  \label{fig4}
\end{figure}
\subsection{Example 3: Frequency-Limited Gramians}

This example first considers a moderate-size problem for which the computation of $F_\Omega(E,A)E^{-1}B$ is feasible, allowing the approximation error to be evaluated. The problem size is then increased to demonstrate the computational efficiency of the G-LRCF-ADI method in a large-scale setting.

For the moderate-size example,
\[
E\in\mathbb{R}^{5000\times 5000},
\qquad
A\in\mathbb{R}^{5000\times 5000},
\qquad
B\in\mathbb{R}^{5000\times 2}.
\]
The parameters are set to $n_c=5$ and $n_d=3$, and the tolerance is $\tau=10^{-8}$. The frequency interval is
\[
\Omega=
\left(
-\frac{\pi}{10},
\frac{\pi}{10}
\right)
\ \text{rad/sec}.
\]
The G-LRCF-ADI method converges in $0.3220$ seconds, with
\[
\frac{\|G_\Omega^{(k)}-G_\Omega^{(k-1)}\|_F}
{\|G_\Omega^{(k-1)}\|_F}
=
2.3717\times10^{-10},
\qquad
\frac{
\|F_\Omega^{(k)}M_\Omega^{(k)}(F_\Omega^{(k)})^\top\|_F
}{
\|EG_\Omega^{(k)}B^\top+
B(G_\Omega^{(k)})^\top E^\top\|_F
}
=
0~(\text{below machine precision}).
\]
The shifts $\alpha_i=-e^{j\nu_i}$ are selected as $12$ points on the unit circle with $\nu_i\in\Omega$. The poles $\beta_i$ consist of $10$ poles $\zeta_i$ and $2$ poles $\eta_i$.

For comparison, computing $F_\Omega(E,A)E^{-1}B$ directly requires $92.9014$ seconds, even for a model of order $5000$. The relative approximation error and the relative residual are
\begin{align}
\frac{
\|F_\Omega(E,A)E^{-1}B-G_\Omega^{(k)}\|_F
}{
\|F_\Omega(E,A)E^{-1}B\|_F
}
&=
9.4004\times10^{-11},
\nonumber\\
\frac{
\|A\hat{P}_\Omega^{(k)}A^\top
-E\hat{P}_\Omega^{(k)}E^\top
+EF_\Omega(E,A)E^{-1}BB^\top
+B(F_\Omega(E,A)E^{-1}B)^\top E^\top\|_F
}{
\|EF_\Omega(E,A)E^{-1}BB^\top
+B(F_\Omega(E,A)E^{-1}B)^\top E^\top\|_F
}
&=
0~(\text{below machine precision}).\nonumber
\end{align}
Thus, G-LRCF-ADI produces an accurate approximation in a fraction of a second without directly computing $F_\Omega(E,A)E^{-1}B$.

Next, a large-scale example is considered with
\[
E\in\mathbb{R}^{10^6\times 10^6},
\qquad
A\in\mathbb{R}^{10^6\times 10^6},
\qquad
B\in\mathbb{R}^{10^6\times 2}.
\]
The shifts $\alpha_i=-e^{j\nu_i}$ are selected as $16$ points on the unit circle with $\nu_i\in\Omega$. The poles $\beta_i$ consist of $10$ poles $\zeta_i$ and $6$ poles $\eta_i$. The G-LRCF-ADI method converges in $5.9853$ seconds, with
\[
\frac{\|G_\Omega^{(k)}-G_\Omega^{(k-1)}\|_F}
{\|G_\Omega^{(k-1)}\|_F}
=
1.8076\times10^{-10},
\qquad
\frac{
\|F_\Omega^{(k)}M_\Omega^{(k)}(F_\Omega^{(k)})^\top\|_F
}{
\|EG_\Omega^{(k)}B^\top+
B(G_\Omega^{(k)})^\top E^\top\|_F
}
=
7.9158\times10^{-9}.
\]
These results show that G-LRCF-ADI computes an accurate frequency-limited Gramian approximation efficiently, including for large-scale models.
\subsection{Example 4: Data-Driven BT}

In this example, the state-space matrices have dimensions
\[
E\in\mathbb{R}^{200\times 200},
\qquad
A\in\mathbb{R}^{200\times 200},
\qquad
B\in\mathbb{R}^{200\times 3},
\qquad
C\in\mathbb{R}^{2\times 200}.
\]
The parameters are set to $n_c=25$ and $n_d=10$. The state-space realization is used only to generate the samples $G(e^{j\nu_i})$. After obtaining these samples, the state-space matrices are not accessed, and BT is performed non-intrusively.

Figure \ref{fig5} shows the shifts used in this example. The shifts $\alpha_i$ and $\hat{\alpha}_i$ lie on the unit circle, allowing the required frequency-response samples to be measured experimentally. The shifts $\beta_i$ and $\hat{\beta}_i$ lie inside the unit circle, as required by the G-LRCF-ADI method.

\begin{figure}[!h]
  \centering
  \includegraphics[width=10cm]{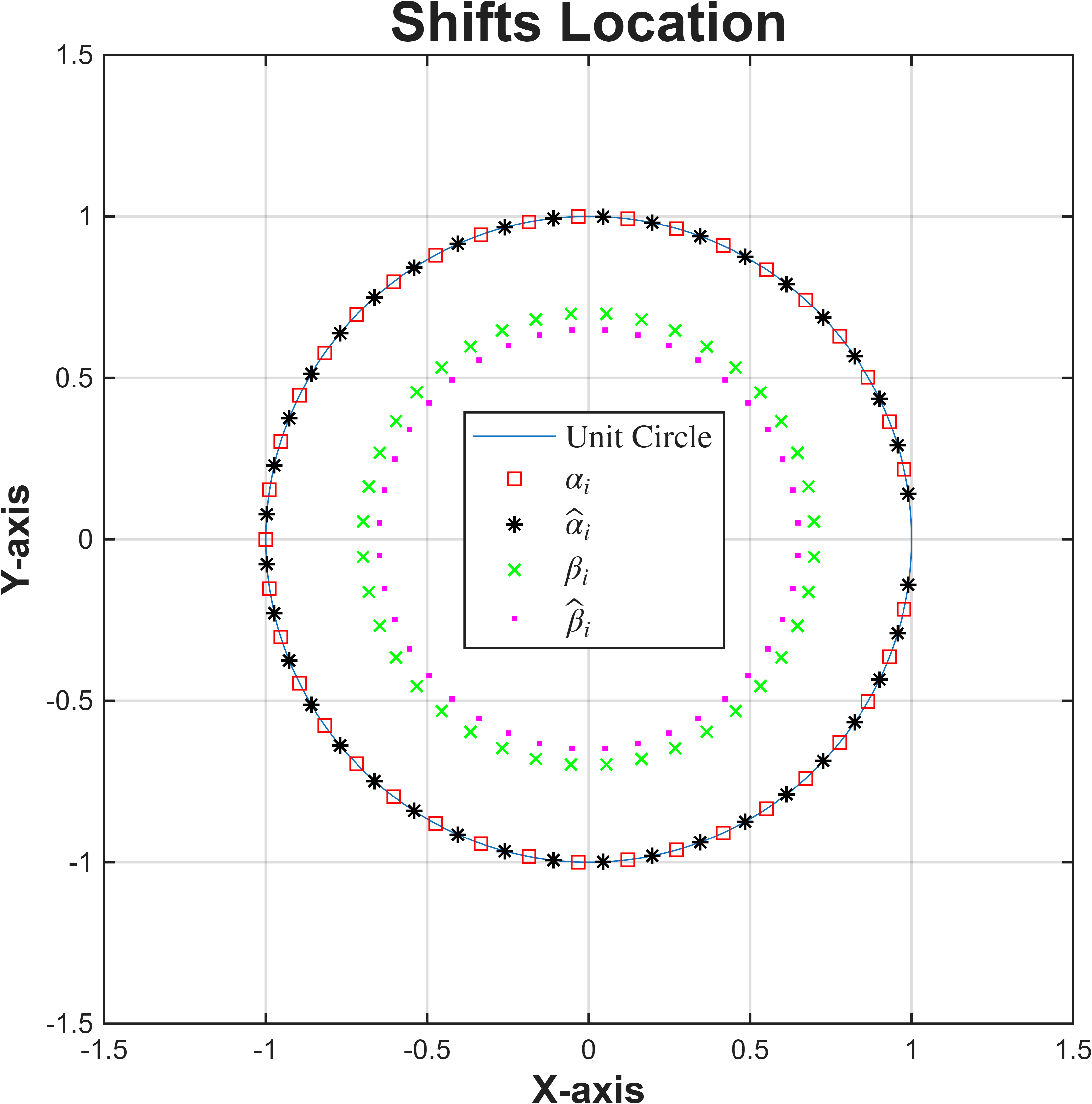}
  \caption{Shift locations}
  \label{fig5}
\end{figure}

A $20$th-order ROM is constructed using the G-LRCF-ADI-based data-driven BT method. Its Hankel singular values are compared with those of the ROM obtained by intrusive BT. As shown in Figure \ref{fig6}, the data-driven method captures all $20$ dominant Hankel singular values of the original model.

\begin{figure}[!h]
  \centering
  \includegraphics[width=10cm]{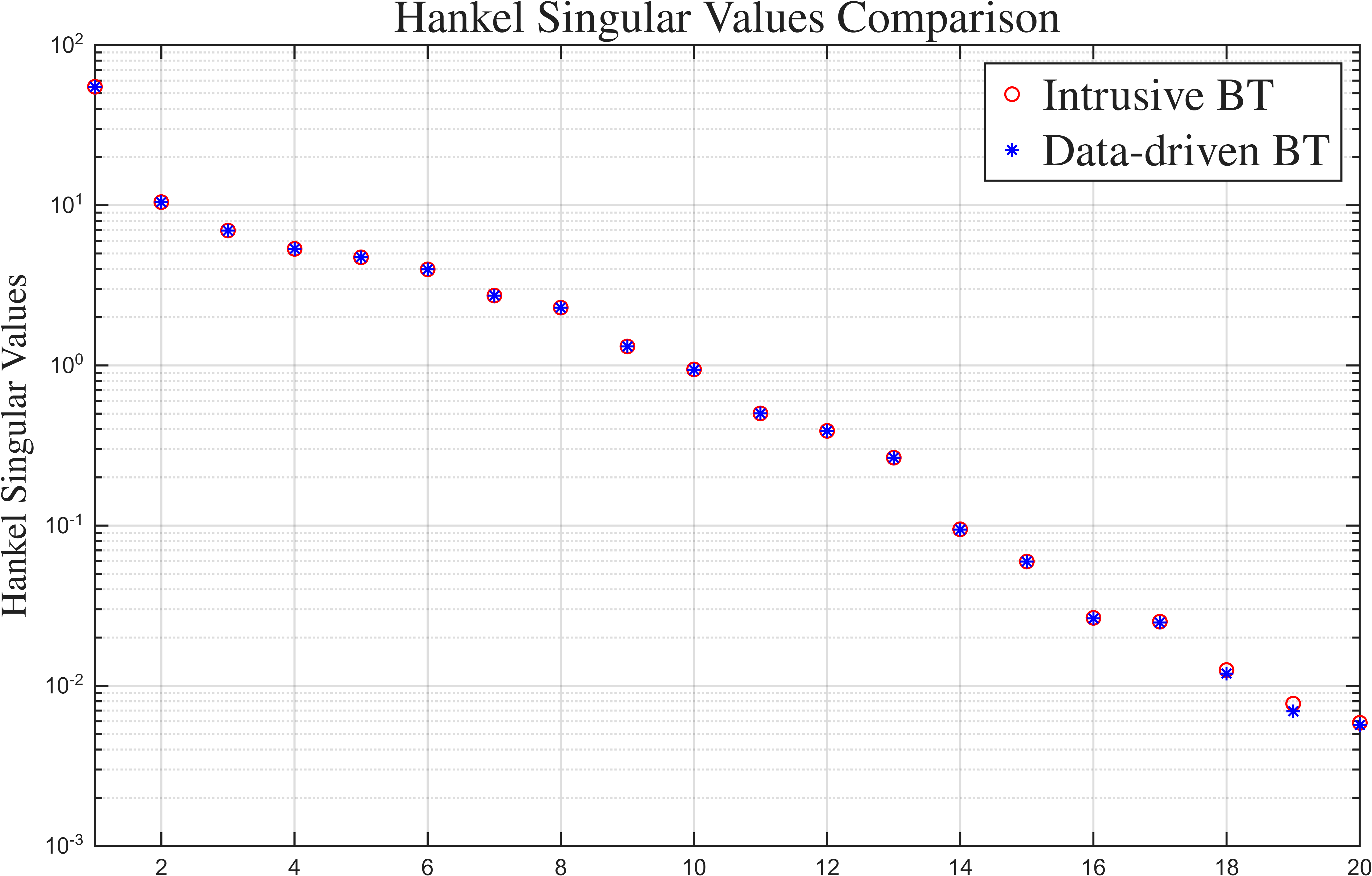}
  \caption{Comparison of Hankel singular values}
  \label{fig6}
\end{figure}
\subsection{Example 5: Data-Driven FLBT}

Consider the model from the previous example, with the frequency interval set to
\[
\Omega=
\left(
-\frac{\pi}{4},
\frac{\pi}{4}
\right)
\ \text{rad/sec}.
\]
Figure \ref{fig7} shows the shifts used in this example. The shifts $\alpha_i$ and $\hat{\alpha}_i$ lie on the unit circle on the opposite side of $\Omega$, since G-LRCF-ADI interpolates at the mirror images of the shifts. Consequently, the resulting projected model interpolates $G(z)$ over the frequency interval $\Omega$ for implicitly approximating the frequency-limited Gramians.

As in the previous example, the state-space matrices are used only to generate samples $G(e^{j\nu_i})$. After obtaining these samples, the state-space realization is not accessed, and FLBT is performed non-intrusively.

\begin{figure}[!h]
  \centering
  \includegraphics[width=10cm]{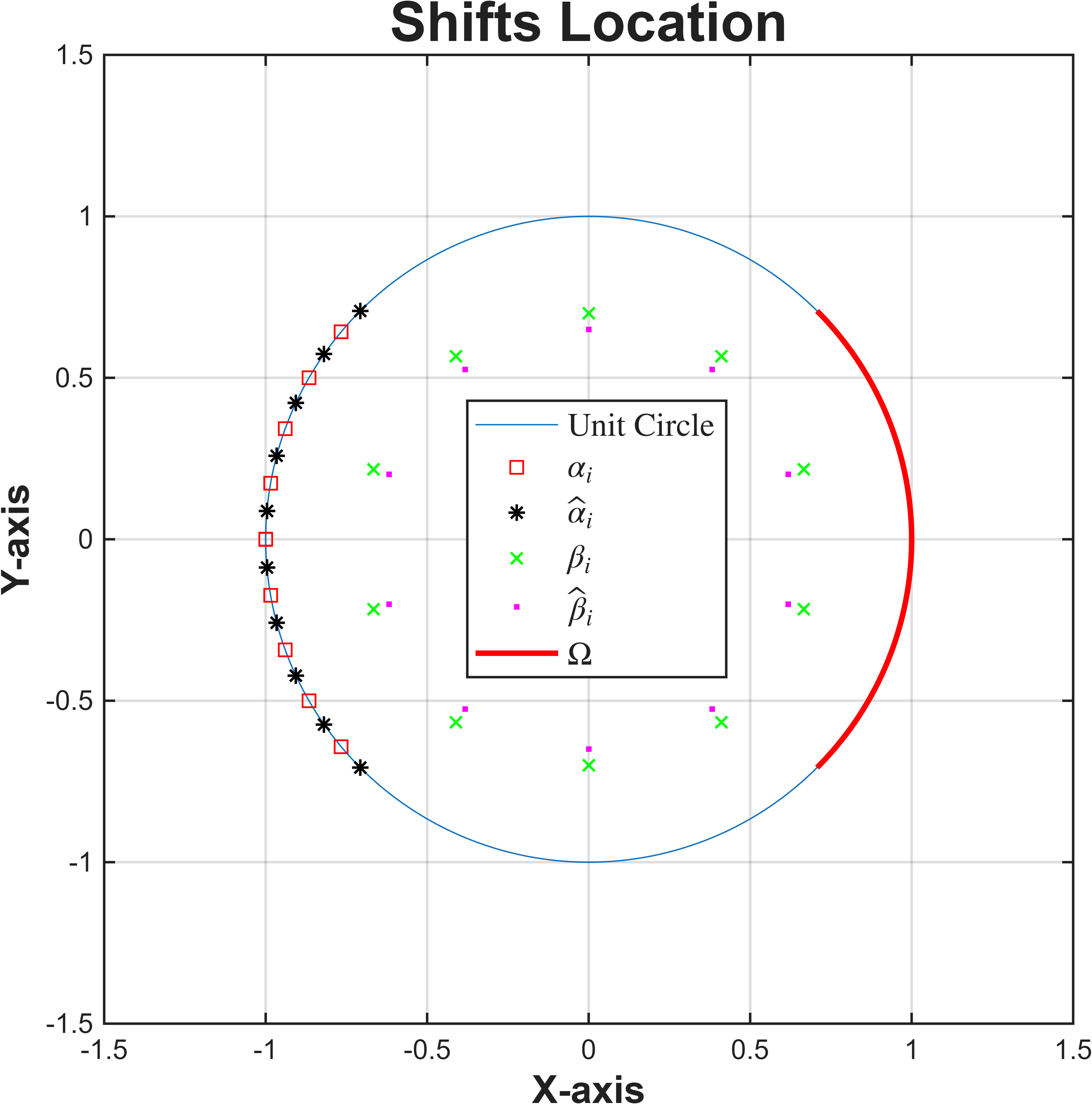}
  \caption{Shift locations}
  \label{fig7}
\end{figure}

A $20$th-order ROM is constructed using G-LRCF-ADI-based data-driven FLBT. Its frequency-limited Hankel singular values are compared with those of the ROM obtained using intrusive FLBT. As shown in Figure \ref{fig8}, the data-driven and intrusive FLBT methods produce nearly identical results.

\begin{figure}[!h]
  \centering
  \includegraphics[width=10cm]{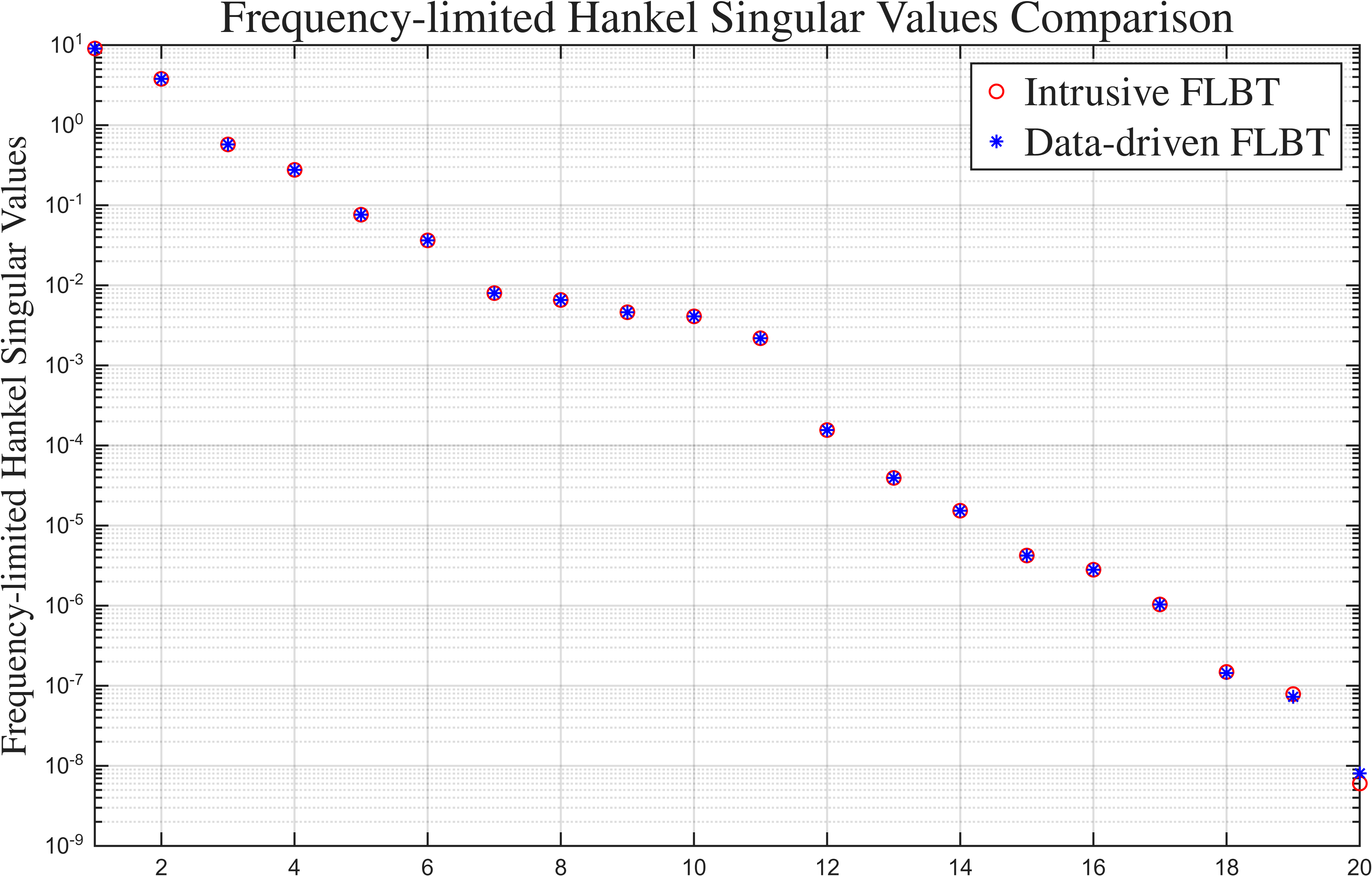}
  \caption{Comparison of frequency-limited Hankel singular values}
  \label{fig8}
\end{figure}
\section{Conclusion}

This paper considers low-rank solutions of large-scale Stein equations. The low-rank Cholesky factor ADI method for Stein equations is revisited, and previously unknown properties are established. In particular, the method is shown to implicitly perform $\mathcal{H}_2$-pseudo-optimal MOR, although the ROM matrices are not explicitly constructed. An automatic shift-generation method is also proposed. Numerical results show that it produces effective shifts and yields rapid residual decay without user intervention.

The low-rank Cholesky factor ADI method is further generalized to allow ADI shifts anywhere in the complex plane, rather than only outside the unit circle. The generalized method is applied to frequency-limited Stein equations. Furthermore, a data-driven balanced truncation method based on the generalized low-rank solver is proposed. It computes ROMs non-intrusively from measurable transfer function samples on the unit circle, without access to a state-space realization. Numerical experiments on large-scale state-space models demonstrate the effectiveness of the proposed methods as Stein equation solvers and data-driven BT methods.
\section*{Appendix}
\begin{proof}
We suppress the factors $\otimes I_m$ and prove the scalar block case. Write $\hat T^{(k)}=T^{(k)}\otimes I_m$. Set
\[
\delta_i=\alpha_i+\beta_i,
\qquad
a_i=\frac1{\beta_i}-\overline{\beta_i},
\qquad
D^{(k)}=Z^{(k)}(Z^{(k)})^*.
\]
Thus
\[
D^{(k)}=\operatorname{diag}(d_1,\dots,d_k),
\qquad
d_i=\frac{1-|\beta_i|^2}{\prod_{j=1}^i|\beta_j|^2}.
\]

After congruence by $X_{\mathrm{fadi}}^{(k)^{-1}}$, equation \eqref{proj_lyap} becomes
\[
A_kP_kA_k^*-P_k+\mathbf 1_k\mathbf 1_k^*=0,
\]
where
\[
(A_k)_{ij}=
\begin{cases}
\beta_i, & i=j,\\
\delta_j, & j<i,\\
0, & j>i.
\end{cases}
\]
It suffices to show that
\[
P_k=T^{(k)*}D^{(k)}T^{(k)}
\]
satisfies this equation.

For $k=1$,
\[
T^{(1)}=\frac1{a_1},
\qquad
d_1=\frac{1-|\beta_1|^2}{|\beta_1|^2},
\]
so
\[
T^{(1)*}D^{(1)}T^{(1)}
=
\frac1{1-|\beta_1|^2},
\]
which solves $\beta_1P_1\overline{\beta_1}-P_1+1=0$.

For $k=2$, write
\[
T^{(2)}=
\begin{bmatrix}
T^{(1)} & u\\
0 & \tau
\end{bmatrix},
\qquad
u=T^{(2)}_{12},
\quad
\tau=T^{(2)}_{22}.
\]
Since $T^{(1)*}d_1=1/\beta_1$, the upper-right block of the Stein equation gives
\[
\left(\frac1{\beta_1}-\overline{\beta_2}\right)u
=
1+\overline{\delta_1}\,T^{(1)}_{11},
\]
which is the recursion for $u$. The remaining $(2,2)$ residual is
\[
R_{22}
=
|\delta_1|^2P_{11}
+
\delta_1\overline{\beta_2}P_{12}
+
\beta_2\overline{\delta_1}P_{21}
+
|\beta_2|^2P_{22}
-
P_{22}
+
1.
\]
Substituting
\[
P_{12}=\frac{u}{\beta_1},
\qquad
P_{22}=d_1|u|^2+d_2|\tau|^2
\]
and using the preceding equation for $u$ yields
\[
R_{22}
=
\frac{|1-a_1u|^2-|a_2\tau|^2}{|\beta_1|^2}.
\]
The recursion for $\tau$ is $a_2\tau=1-a_1u$, hence $R_{22}=0$.

For $k=3$, write
\[
T^{(3)}=
\begin{bmatrix}
T^{(2)} & u\\
0 & \tau
\end{bmatrix},
\qquad
u=
\begin{bmatrix}
u_1\\u_2
\end{bmatrix}
=
\begin{bmatrix}
T^{(3)}_{13}\\T^{(3)}_{23}
\end{bmatrix},
\quad
\tau=T^{(3)}_{33}.
\]
Let $P_2=T^{(2)*}D^{(2)}T^{(2)}$. The upper-right block of the Stein equation is
\[
(I-\overline{\beta_3}A_2)T^{(2)*}D^{(2)}u
=
A_2P_2
\begin{bmatrix}
\overline{\delta_1}\\
\overline{\delta_2}
\end{bmatrix}
+
\mathbf 1_2.
\]
Substituting the explicit $2\times2$ matrices $A_2,T^{(2)},D^{(2)}$ and simplifying gives the equivalent triangular system
\[
\left(\frac1{\beta_1}-\overline{\beta_3}\right)u_1
=
1+\overline{\delta_1}\,T^{(2)}_{11}
+\overline{\delta_2}\,T^{(2)}_{12},
\]
\[
a_1u_1+
\left(\frac1{\beta_2}-\overline{\beta_3}\right)u_2
=
1+\overline{\delta_2}\,T^{(2)}_{22}.
\]
These are exactly the recursions for $[t_1^{(3)}]_1$ and $[t_1^{(3)}]_2$. With these two equations imposed, the $(3,3)$ residual reduces to
\[
R_{33}
=
\frac{
|1-a_1u_1-a_2u_2|^2
-
|a_3\tau|^2
}{
|\beta_1|^2|\beta_2|^2
}.
\]
The recursion for $\tau=t_2^{(3)}$ is
\[
a_3\tau=1-a_1u_1-a_2u_2,
\]
so $R_{33}=0$. Hence the claim holds for $k=3$.

The same block calculation gives the induction step. Assume the claim holds for $k-1$. The upper-right block of the $k$th Stein equation is equivalent, by forward substitution, to
\[
\left(\frac1{\beta_i}-\overline{\beta_k}\right)[t_1^{(k)}]_i
+
\sum_{r=1}^{i-1}
a_r[t_1^{(k)}]_r
=
1+
\sum_{r=i}^{k-1}
\overline{\delta_r}\,T^{(k-1)}_{ir},
\qquad i=1,\dots,k-1,
\]
which is the first recursion in the theorem. After these equations are satisfied, the final diagonal residual is
\[
R_{kk}
=
\frac{
\left|1-\sum_{r=1}^{k-1}a_r[t_1^{(k)}]_r\right|^2
-
|a_k t_2^{(k)}|^2
}{
\prod_{j=1}^{k-1}|\beta_j|^2
}.
\]
The second recursion in the theorem gives
\[
a_k t_2^{(k)}
=
1-\sum_{r=1}^{k-1}a_r[t_1^{(k)}]_r,
\]
so $R_{kk}=0$. Thus the Stein residual vanishes for all $k$.

Returning to the original coordinates,
\[
\tilde P^{(k)}
=
X_{\mathrm{fadi}}^{(k)}
T^{(k)*}D^{(k)}T^{(k)}
X_{\mathrm{fadi}}^{(k)*}
=
X_{\mathrm{fadi}}^{(k)}
(\hat T^{(k)})^*
Z^{(k)}(Z^{(k)})^*
\hat T^{(k)}
X_{\mathrm{fadi}}^{(k)*}
=
\tilde Z^{(k)}(\tilde Z^{(k)})^*.
\]
\end{proof}

\end{document}